\documentclass[11pt]{article}\textwidth 160mm\textheight 235mm
\usepackage{amsfonts}
\usepackage{amssymb}
\usepackage{graphicx}
\usepackage{amsmath}
\usepackage{amsthm}
\usepackage{enumitem}
\usepackage{tikz}
\usepackage{cancel}
\usepackage{nicematrix}
\usepackage{todonotes}
\usepackage{showlabels}
\usetikzlibrary{matrix}
\usepackage{mathrsfs}
\usetikzlibrary{arrows,shapes}
\usepackage{cite}
\usepackage{hyperref}
\hypersetup{colorlinks=true, urlcolor=blue}
\expandafter\let\expandafter\oldproof\csname\string\proof\endcsname
\let\oldendproof\endproof
\renewenvironment{proof}[1][\proofname]{%
  \oldproof[\ttfamily \scshape \bf #1. ]%
}{\oldendproof}
\def\O{{\bf O}}

\def\S{{\bf {S}}}

\def\B{{\bf B}}
\def\R{{\bf {R}}}
\def\N{{\bf {N}}}
\def\ox{\bar{x}}
\def\oy{\bar{y}}

\def\ov{\bar{v}}
\def\ow{\bar{w}}

\def\ss{\scriptsize }
\def\diag{\mbox{\rm diag}\,}

\def\ve{\varepsilon}

\def\X{{\bf X}}
\def\P{{\bf P}}
\def\Y{{\bf Y}}

\def\tilde{\widetilde}
\def\dist{{\rm dist}}

\def\Lm{{\Lambda}}
\def\tto{\rightrightarrows}

\def\d{{\rm d}}
\def\l{{\ell}}
\def\sub{\partial}

\def\Hat{\widehat}

\def\Bar{\overline}
\def\ra{\rangle}
\def\la{\langle}
\def\ve{\varepsilon}

\def\ri{\mbox{\rm ri}\,}

\def\gph{\mbox{\rm gph}\,}
\def\epi{\mbox{\rm epi}\,}

\def\dom{\mbox{\rm dom}\,}

\def\diag{\mbox{\rm diag}\,}

\def\dn{\downarrow}

\def\ph{\varphi}

\def\oR{\Bar{\bf{R}}}

\def\lm{\lambda}

\def\dd{\delta}
\def\al{\alpha}

\def\th{\theta}

\def\sm{\hbox{${1\over 2}$}}

\def\sce{\setcounter{equation}{0}}
\def\Diag{\mbox{\rm Diag}\,}
\def\diag{\mbox{\rm diag}\,}
\def\prox{\mbox{\rm prox}}
\def \b{{\}_{k\in\N}}}

\DeclareMathOperator*{\argmin}{argmin}

\def\verl{ \;\rule[-0.4mm]{0.2mm}{0.27cm}\;}

\begin{document}
\vspace*{0.5in}
\begin{center}
{\Large\bf Twice Epi-Differentiability  of Spectral Functions and its applications}\\[1 ex]
CHAO DING\footnote{State Key Laboratory of Mathematical Sciences, Academy of Mathematics and Systems Science, Chinese Academy of Sciences, Beijing 100190, China; School of Mathematical Sciences, University of Chinese Academy of Sciences, Beijing 100049, China; Institute of Applied Mathematics, Academy of Mathematics and Systems Science, Chinese Academy of Sciences, Beijing 100190, China (dingchao@amss.ac.cn). Research of this author is supported in part by the National Key R\&D Program of China (No.~2021YFA1000300, No.~2021YFA1000301), National Natural Science Foundation of China (No.~12531014), and CAS Project for Young Scientists in Basic Research (No.~YSBR-034). }\quad
EBRAHIM SARABI\footnote{Department of Mathematics, Miami University, Oxford, OH 45065, USA (sarabim@miamioh.edu). Research of this author is partially supported by the U.S. National Science Foundation under the grant DMS 2108546.}\quad
SHIWEI WANG\footnote{Institute of Applied Mathematics, Academy of Mathematics and Systems Science, Chinese Academy of Sciences, Beijing,
China. This work is done while the author is a visiting scholar at the Institute of Operational Research and Analytics, National
University of Singapore, Singapore (wangshiwei@amss.ac.cn).}
\end{center}
\vspace*{0.05in}
\small{\bf Abstract.} Second-order variational properties have been shown to play important theoretical and numerical roles for different classes of optimization problems. Among such properties, twice epi-differentiability has a special place because of its ubiquitous presence in various classes of extended-real-valued functions that are important for optimization problems. We provide a useful characterization of this property for spectral functions by demonstrating that it can be characterized via the same property of the symmetric part of the spectral representation  of an eigenvalue function. Our approach allows us to bypass the rather restrictive convexity assumption, used in many recent works that targeted second-order variational properties of spectral functions.  By this theoretical tool, several applications on the proto-differentiability of subgradient mappings, the
directional differentiability of the proximal mapping of spectral functions are achieved. We finally use our established theory to study twice epi-differentiability of leading eigenvalue functions and practical regularization terms that have important applications in statistics and the robust PCA.   
\\[1ex]

\newtheorem{Theorem}{Theorem}[section]
\newtheorem{Proposition}[Theorem]{Proposition}
\newtheorem{Remark}[Theorem]{Remark}
\newtheorem{Lemma}[Theorem]{Lemma}
\newtheorem{Corollary}[Theorem]{Corollary}
\newtheorem{Definition}[Theorem]{Definition}
\newtheorem{Example}[Theorem]{Example}
\newtheorem{Algorithm}[Theorem]{Algorithm}
\renewcommand{\theequation}{{\thesection}.\arabic{equation}}
\renewcommand{\thefootnote}{\fnsymbol{footnote}}
\newcommand{\bigzero}{\mbox{\normalfont\Large\bfseries 0}}
\normalsize

\section{Introduction}\sce
In this paper, we study the second-order variational properties of orthogonally invariant functions (also called spectral functions \cite{F81}). A mapping 
$g:\S^n\to \oR:=[-\infty,\infty]$, is orthogonally invariant if
$$g(X)=g(U^\top XU),\;\mbox{for all}\;X\in\S^n\;\mbox{and for any}\; U\in\O^n,$$
where $\S^n$ stands for the real vector space of $n\times n$ symmetric matrices and $\O^n$ is the set of all real $n\times n$ orthogonal matrix. It is known (see \cite{l962}) that any spectral function admits the composite representation
\begin{equation}\label{spec}
g(X)=(\th\circ\lm)(X), \;\; X\in \S^n,
\end{equation}
where $\th:\R^n\to  \oR$ is a symmetric (permutation-invariant) function on $\R^n$, i.e., for all $x\in\R^n$ and any $n\times n$ permutation matrix $P$, $\theta(x)=\theta(Px)$,  where $\lm(X):=\big(\lm_1(X),\ldots,\lm_n(X)\big)$ denotes the vector of eigenvalues of $X$ arranged in nonincreasing order $\lm_1(X)\geq\ldots\geq\lm_n(X)$. Spectral functions are widely used in areas such as neural networks, signal processing, robust PCA, and graph theory, among others \cite{KLV97,PW18,GD14,EY17,dp}. They also cover many classical models, including semidefinite programming problems (SDP) \cite{ZST} and eigenvalue optimization \cite{OM}. Over the past three decades, extensive theoretical foundations have been developed for spectral functions; see, e.g., \cite{ls,l962,QY,dst1,dst2,d12}. 

While prior studies have delivered notable advances in first-order variational properties of spectral functions and algorithm design for particular classes of optimization problems such as CMatOP \cite{cd}, the existing second-order variational theory of spectral functions remains limited in light of rapid progress in machine learning and statistical optimization. 
To narrow this gap, we aim to present a comprehensive analysis of twice epi-differentiability of spectral functions that are not necessarily convex. 
Twice epi-differentiability, introduced in \cite{r90}, provides second-order approximations to the epigraphs of extended-real-valued functions. It also can be used to establish the second-order optimality conditions for different class of optimization problems \cite{hs}.  On the algorithmic side, recent results in \cite{mms1,ms,HS24s} highlight a fundamental  role that  twice epi-differentiability plays  for establishing convergence of numerical algorithms  including generalized Newton methods \cite{mors} and the augmented Lagrangian method \cite{hs24} that are widely used in large-scale optimization.  

Twice epi-differentiability has been explored for important classes of functions such as piecewise linear quadratic (cf. \cite[Proposition~13.9]{rw}), ${C}^2$ cone reducible convex sets (cf. \cite[Theorem~6.2]{mms1}), and more generally parabolically regular functions (cf. \cite[Theorem~3.8]{ms}).
These results were extended to some classes of eigenvalue functions by 
Torki  in \cite[Theorem 3.2]{t2}), where he showed for the first time that some eigenvalue functions enjoy this property as well. His result, however, 
fell short of providing a systematic approach to characterizing twice epi-differentiability of eigenvalue functions through the same property of their symmetric parts, meaning the function $\th$ in \eqref{spec},  a pattern which was pioneered by Davis in \cite{da} for convexity of spectral functions and was remarkably extended to various first-order variational properties by Lewis; see  \cite{l99}.  It is worth mentioning that a similar pattern was demonstrated to hold for some second-order variational properties  of  spectral functions such as twice differentiability (cf. \cite{ls}), prox-regularity (cf. \cite{dlms}), ${\cal C}^2$-cone reducibility (cf. \cite{cdz}),  and parabolic regularity (cf. \cite{ms2,HKS}). 

This paper continues the path, initiated recently in \cite{ms2}, to conduct second-order variational analysis of spectral functions.  We should add here that the latter work focused mainly on the characterization of parabolic regularity, which is a strictly stronger property than twice epi-differentiability,  of spectral functions. Moreover, convexity was assumed in  the main results of \cite{ms2}. Our results have two major differences with \cite{ms}. First and foremost, we study a different second-order variational property than the one in \cite{ms2}. Second, we improve some of the results therein that allow us to drop convexity in our developments in this paper. 
This opens the door to possibly remove convexity from the main assumptions in \cite{ms2}, which will be a subject for our future research. 
Below, we summarize our main contributions: 
\begin{itemize}
    \item We show that   twice epi-differentiability  of the spectral function $g$ in \eqref{spec} can be characterized  via the same property of its symmetric part $\theta$ even when $g$ is nonconvex. As a consequence of this characterization, we also calculate the second subderivative of spectral functions. A similar result is proven for   twice semidifferentiability of differentiable spectral functions. Note that the principal challenge in the nonconvex case is that subgradients $\partial\theta$ need not have components in nonincreasing order, which complicates compatibility with the ordered eigenvalue map.  These brand new results provide important insights in how to study second-order variational properties of nonconvex spectral functions. 
    
    \item Based on these inheritance results, the proto-differentiability
of subgradient mappings and the directional differentiability of the proximal mapping of spectral functions are characterized. Furthermore, we provide a characterization of the concept of generalized twice differentiability of spectral functions, which was introduced recently in \cite{r25}. 

    \item As a consequence of our second order variational developments, we study  variational properties of nonconvex spectral models, e.g., leading eigenvalue-based penalties and MCP penalties on eigenvalue values, which are widely used in modern statistics and machine learning   yet lack rigorous analysis. We further propose a unified framework that treats such terms without requiring a case-by-case argument. 
\end{itemize}

The remaining parts of this paper are organized as follows. In Section \ref{sect02}, we introduce the notation used throughout and reviews some preliminaries on eigenvalue and spectral functions. 
Section \ref{sect05} presents a characterization of twice epi-differentiability of spectral functions via the same property of their symmetric parts.    Several theoretical applications of this characterization   are provided in  Section \ref{sec:aptt}. 
Section \ref{secExp} illustrates some applications of our established theory  in important examples, including the leading eigenvalues, MCP penalties on singular values, and several commonly used statistical regularizers. We conclude our paper in the
final section.

\section{Notations and Preliminaries}\sce \label{sect02}
In this section, we mainly introduce some notations used throughout this paper,  some preliminary knowledge on the eigenvalue
function and spectral functions. 
Suppose $\X$ and $\Y$ are Euclidean spaces.   
For any set $C\subset \X$, we use $\dd_C$ to denote its indicator function. $\dist (x,C)$ is defined as the distance between $x\in \X$ and set $C$. 
$\R^{m\times n}$ is the space of all real $n\times m$   matrices. 
For any $Z\in\R^{m\times n}$ and given two index sets $\alpha\subseteq  \{1,\ldots, m\}$, $\beta\subseteq \{1,\ldots, n\}$,  we use $ Z_{\alpha}$ to denote the sub-matrix of $Z$ obtained by removing all the columns of $Z$ not in $\alpha$. $Z_{\alpha\beta}$ denotes the $|\alpha|\times|\beta|$ sub-matrix of $Z$ obtained  by removing all the rows of $Z$ not in $\alpha$ and all the columns of $Z$ not in  $\beta$, where $|\alpha|$ is cardinality of the index set $\alpha$. 
In particular, we use $Z_{j}$ to represent the $j$-th column of $Z$, $j=1,\ldots,n$. 
$Z^{\dagger}$ is the Moore-Penrose generalized inverse of $Z$.  
$\N$ stands for the set of natural numbers.  For any vector $x=(x_1,\ldots,x_n) \in \R^n$,  denote by $\Diag (x)$ the diagonal matrix whose   $i$-th diagonal entry is   $x_i$ for any $i=1,\ldots,n$.

Given $X \in \S^n$, we use $\mu_{1}(X) >  \cdots> \mu_{r}(X) $ to denote the distinct eigenvalues of $X$ and define the index sets 
\begin{equation}\label{index}
\alpha_{m} :=\big\{  i \in  \{ 1,\ldots, n\}   | \; \lambda_{i} (X)  = \mu_{m}(X)\big\},\quad  m=1,\ldots,r.
\end{equation}
Let $\O^n (X)$ be the set of all orthogonal matrices $U$  satisfying the eigenvalue decomposition of $X\in\S^n$, i.e., 
\begin{equation}\label{specdocom}
X  =  U \Diag (\lambda (X)) U^\top.
\end{equation}
Define $\l_i (X)$ for any $i \in \{1, \ldots, n\} $ to be the number of eigenvalues of $X$ that are equal to $\lambda_{i} (X)$
but are ranked before $\lm_i(X)$ including $\lambda_{i} (X)$. This integer allows us to locate $\lm_i(X)$
in the group of the eigenvalues of  $X$ as follows:
\begin{equation}\label{ell}
\lm_1(X)\ge \cdots \ge \lm_{i-\ell_i(X)}>\lm_{i-\ell_i(X)+1}(X)= \cdots =\lm_i(X)\ge \cdots\ge \lm_n(X).
\end{equation}
For simplicity,   we  drop $X$ from  $\ell_i (X)$ when the dependence of $\l_i$ on $X $ can be seen from the context. 
For any function $f:\X\rightarrow\R$, we use $f'(x;h)$ to denote its directional derivative at $x$ for $h$ (cf. \cite[7(17)]{rw}). 
The following first-order expansion of the eigenvalue function is given in \cite[Proposition~1.4]{t1}, which is of great importance in the derivation of the main result of this paper. 

\begin{Proposition}[First-order expansion of eigenvalue functions]\ \label{olemma}
Assume that  $X \in \S^n$ has the eigenvalue decomposition \eqref{specdocom} for some $U \in \O^n (X)$.  Let 
 $\mu_{1} >  \cdots> \mu_{r} $ be distinct eigenvalues of $X$. 
Then, for any  $H\in \S^n$ that $H\to 0$ and any $i\in \{1,\ldots,n\}$, we have 
\begin{equation}\label{secondexp1}
\lambda_{i} (X +  H) = \lambda_{i} (X ) +  \lambda_{\l_i} \big( U_{\al_m}^\top  H U_{\al_m}  + U_{\al_m}^\top  H ( \mu_m I-X)^{\dagger} H U_{\al_m}  \big) + O(\|  H \|^3)
\end{equation}
and 
\begin{equation}\label{fexpan}
\lambda_{i} \big( X + H  \big) = \lambda_{i} (X)  +  \lambda_{\l_i} (U_{\al_m}^{\top} H U_{\al_m}) +  O (\| H \|^2 ),   
\end{equation}
where    $m\in \{1,\ldots,r\}$ with  $i \in \alpha_m$. Thus, we have  
$$
 \lambda^{'} ( X ; H )=\big(\lambda (U_{\al_1}^{\top} H U_{\al_1}),\ldots,\lambda (U_{\al_r}^{\top} H U_{\al_r}) \big). 
 $$
\end{Proposition}

Next, we review some concepts from variational analysis that will be used extensively in this paper. Given a nonempty set $C\subset\X$ with $\ox\in C$, the  tangent cone $T_ C(\ox)$ to $C$ at $\ox$ is defined by
\begin{equation*}\label{2.5}
T_C(\ox)=\big\{w\in\X|\;\exists\,t_k{\dn}0,\;\;w_k\to w\;\;\mbox{ as }\;k\to\infty\;\;\mbox{with}\;\;\ox+t_kw_k\in C\big\}.
\end{equation*}
The regular normal cone of $C$ at $\bar{x}$ is defined as $\widehat{N}_C(\bar{x})=T_C(\ox)^*$. The (limiting) normal cone ${N}_C(\bar{x})$ is defined as $\{v\in \X\mid \exists\; x_k \rightarrow\ox, x_k\in C\;\mbox{and}\;v_k\rightarrow v\; \mbox{with}\;v_k\in\widehat{N}_C(x_k)\}$. The set of all subgradients of $f$ at $\ox$ is defined as $\partial f(\ox)=\{v\in \X\mid\,(v,-1)\in N_{ \ss\epi f}(\ox,f(\ox))\}$ and is called the (limiting) subdifferential of $f$ at $\ox$. Similarly, we can define the regular subdifferential of $f$ at $\ox$, denoted by $\Hat \sub f(\ox)$. Below, we record  a result, established in \cite[Theorem~6]{l99}, that provides a nice formula  for the subdifferential of spectral functions. 

\begin{Proposition} \label{subsp} Assume that  $\th:\R^n\to \oR$ is a proper and  lower semicontinuous \rm{(}lsc\rm{)}  symmetric  function. Then we have 
$$
 \sub (\th\circ\lm)(X) =\big\{U\Diag(v)U^\top|\; v\in \sub \th(\lm(X)), \; U\in \O^n(X)\big\}.
$$
A similar conclusion holds for the regular subdifferential of $\th\circ\lm$.
\end{Proposition}

Given a function $f:\X \to \oR$, its domain is defined  by 
$\dom f =\big\{ x \in \X|\; f(x) < \infty \big \}$.
The  function $f$ is called locally Lipschitz continuous around $\ox $ {\em relative} to $C \subset \dom f$
with constant $\ell \ge0 $
if  $\ox \in C$ with $f(\ox)$ finite  and there exists  a neighborhood $U$ of $\ox$ such that 
\begin{equation*} \label{lipwrtdomain}
 |f(x)  - f(y ) | \leq  \ell\, \| x - y \|  \quad \mbox{for all }\;       x , y \in U \cap C.    
\end{equation*}
Such a function is called locally Lipschitz continuous relative to $C$ if  it is locally Lipschitz continuous around every $\ox\in C $ relative to $C$. 
 Piecewise linear-quadratic functions (not necessarily convex) and an indicator function of a nonempty set are important examples of 
 functions that are locally Lipschitz continuous relative to their {domains}.

The subderivative function of $f:\X\to \oR$ at $\ox$ with $f(\ox)$ finite, denoted by $\d f(\ox)\colon\X\to \oR$,  is defined for any $\ow\in\X$ by
\begin{equation}\label{semi1}
{\mathrm d}f(\ox)(\ow)=\liminf_{\substack{
   t\dn 0 \\
  w\to \ow
  }} {\frac{f(\ox+tw)-f(\ox)}{t}}.
\end{equation}
Next, we review some preliminary results about the subderivative of the spectral functions. 
Denote $\P^n$ as the set of all $n\times n$ permutation matrices. 
Recall that a function $\th:\R^n\to \oR$ is called symmetric if for every 
$x\in \R^n$ and every  $Q\in\P^n$,  $\th(Qx)=\th(x)$ holds.   
It can be checked directly from \eqref{spec} that $\th$ can be expressed as a composite function in the form 
\begin{equation}\label{spec2}
\th(x)=g\big(\Diag(x)\big) \quad \mbox{for all}\;\; x\in \R^n. 
\end{equation} 
Suppose $X$ has  the eigenvalue decomposition of \eqref{specdocom} with $r$ distinct eigenvalues. 
Denote $\P^n_X:=\{Q=\Diag(P_1,\ldots,P_r)\mid P_m\in \P^{|\al_m|}, \;m=1,\ldots,r\}$. It can be checked directly  that $\P^n_X\subset \P^n$. Moreover, if $Q\in \P^n_X $, we have $Q\lm(X)=\lm(X)$. 
If $\th:\R^n\to \oR$ is a symmetric function and $X\in \S^n$ with $\th(\lm(X))$ finite, 
it is not hard to see  that 
\begin{equation}\label{sub_sym}
\d \th(\lm( X )) (v) =\d \th(\lm( X )) (Qv)
\end{equation}
for any $v\in \R^n$ and any $Q\in \P^n_X$.

Recalling the index sets $\al_m$, $m=1,\ldots,r$,  from \eqref{index}, we  partition a vector $p\in \R^n$ into  $(p_{\al_1},\ldots, p_{\al_r})$, where $p_{\al_m}\in \R^{|\al_m|}$ for any $m=1,\ldots,r$.
Recall   from \cite[Definition~7.25]{rw} that an lsc function $f: \X\to \oR$ is said to be subdifferentially regular at $\ox\in \X$  if    $f(\ox)$ is   finite and $\widehat N_{\ss\epi f}(\ox, f(\ox))= N_{\ss \epi f}(\ox,f(\ox))$. 
Denote by $\R^n_{\dn}$ the set of all vectors $(x_1,\ldots,x_n)$ such that $x_1\ge \cdots\ge x_n$. 
The following result shows the explicit formula of the subderivative of  spectral functions when they are subdifferentially regular. Note that a similar result was recently established   \cite[Theorem 3.5]{ms2} for convex spectral functions. 
 
 \begin{Theorem}[Subderivatives of spectral functions]\label{specsubd}
Let $\th:\R^n\to \oR$ be a symmetric function  and let $X\in \S^n$ with $(\th\circ \lm)(X)$   finite. If $\th$ is  lsc and subdifferentially regular  with $\sub \th(\lm(X))\neq \emptyset$, then for all $H \in \S^n$ we have
\begin{equation}\label{dchain}
\d (\th\circ\lm)( X ) (H) =  \d \theta (\lambda(X)) (\lambda^{'} (X ; H)).
\end{equation}   
\end{Theorem}  
\begin{proof}
The proof of the inequality ``$\ge$" in \eqref{dchain} is exactly the same as that in \cite[Theorem 3.5]{ms2}, so we omit it here for brevity. To justify  the opposite inequality, we can also obtain for any $\varepsilon>0$, we may choose $Y\in\sub (\th \circ \lambda)(X)$ such that
$$    
 \d (\th \circ \lambda)(X) (H) \leq  \ve  + \la Y, H \ra   
$$
by following the same argument as in the proof of \cite[Theorem 3.5]{ms2}. 
Since $Y \in \sub (\th \circ \lambda)(X)$, it follows from Proposition~\ref{subsp} that there are $U\in \O^n(X)$ and $v\in \sub  \th(\lm(X))$ such that $Y=U\Diag(v)U^\top$.  
Since $v=(v_{\al_1}, \ldots, v_{\al_r})$, we can choose a permutation matrix $Q_m\in \P^{|\al_m|}$ such that $\tilde v_{\al_m}=Q_mv_{\al_m}$ with  $\tilde v_{\al_m}\in \R^{|\al_m |}_{\dn}$ for any $m=1,\ldots,r$.
Set $Q= \Diag(Q_1,\ldots,Q_r)$ and use  Fan's inequality \cite{Fan49} to obtain 
\begin{eqnarray*}
	\d (\th \circ \lambda)(X) (H) &\leq&  \ve  + \la Y, H \ra  =   \ve  + \la \Diag(v), U^\top H U \ra  = \ve  +  \sum_{m=1}^{r} \la \Diag(v)_{\al_m\al_m}  ,  U_{\al_m}^\top  H U_{\al_m} \ra\\
	&\leq &  \ve  +  \sum_{m=1}^{r} \la \tilde v_{\al_m}, \lambda ( U_{\al_m}^\top H U_{\al_m}) \ra   \
	=  \ve  +  \sum_{m=1}^{r}\la Q_mv_{\al_m}, \lambda ( U_{\al_m}^\top H U_{\al_m}) \ra    \\
	&=&   \ve  +  \la v  ,  Q^\top \lambda^{'} (X; H) \ra   
	\leq   \ve  +  \d \th ( \lambda(X)) (Q^\top \lambda^{'} (X; H) )\\
	&=&  \ve  +  \d \th ( \lambda(X)) ( \lambda^{'} (X; H) ),
\end{eqnarray*}
where the last inequality results from the fact that $\th$ is  subdifferentially regular, and  $v\in \sub \th(\lm(X))$ and where the last equality results from \eqref{sub_sym}. 
Letting $\varepsilon \dn 0$, we get the opposite inequality in  \eqref{dchain}.
\end{proof}

\section{ Twice Epi-Differentiability of Spectral Functions}\label{sect05}  \sce

The main objective of this section is to provide a characterization of   twice epi-differentiability of the spectral function $g$ in \eqref{spec} via the same property of the symmetric function $\theta$ therein. To do so, we recall some second-order variational properties, important for our developments in this section. 
Given a function $f: \X \to \oR$ and $\ox \in \X$ with $f(\ox)$ finite, denote the parametric  family of 
second-order difference quotients for $f$ at $\ox$ for $\ov\in \X$ as 
\begin{equation*}\label{lk01}
\Delta_t^2 f(\bar x , \ov)(w)=\dfrac{f(\ox+tw)-f(\ox)-t\langle \ov,\,w\rangle}{\frac{1}{2}t^2}\quad\quad\mbox{with}\;\;w\in \X, \;\;t>0.
\end{equation*}
The {second subderivative}  of $f$ at $\ox$ for $\ov$   is defined by 
\begin{equation*}\label{ssd}
\d^2 f(\bar x , \ov)(w)= \liminf_{\substack{
   t\dn 0 \\
  w'\to w
  }} \Delta_t^2 f(\ox , \ov)(w'),\;\; w\in \X.
\end{equation*}
Following \cite[Definition~13.6]{rw}, a function $f:\X \to \oR$ is said to be {twice epi-differentiable} at $\bar x$ for $\ov\in\X$, with $f(\ox) $ finite, 
if the sets $\epi \Delta_t^2 f(\bar x , \ov)$ converge to $\epi \d^2 f(\bar x,\ov)$ as $t\downarrow 0$; see \cite[Chapter~4]{rw} for the definition of a sequence of sets. The latter means by  
\cite[Proposition~7.2]{rw} that for every sequence $t_k\downarrow 0$ and every $w\in\X$, there exists a sequence $w_k \to w$ such that
\begin{equation}\label{dedf}
\d^2 f(\bar x,\ov)(w) = \lim_{k \to \infty} \Delta_{t_k}^2 f(\ox , \ov)(w_k).
\end{equation}

The  second subderivative was introduced in \cite{r85}, where the concept of twice epi-differentiability for convex functions was studied for the first time. 
 The importance of the second subderivative resides in the fact that it can characterize the quadratic growth condition for optimization problems; see \cite[Theorem~13.24]{rw}. So, it is crucial
 for many applications to calculate it in terms of the initial data of an optimization problem. This task was carried out for 
   major classes of functions including the convex piecewise linear-quadratic functions  in  \cite[Proposition~13.9]{rw},  the second-order cone in \cite[Example~5.8]{mms2}, the cone of positive semidefinite symmetric matrices in  \cite[Example~3.7]{ms}, and  the augmented Lagrangian of constrained optimization problems in  \cite[Theorem~8.3]{mms2}.  The second subderivative  can also be leveraged in the study of stability properties of optimization problems and generalized equations. To this end, it is vital to be able to calculate the second subderivative and characterize twice epi-differentiability for important classes of functions in order to study various stability properties of constrained and composite optimization problems. 

We aim in this section to provide a unified approach that simultaneously answers both aforementioned aspects for spectral functions. We begin with the following result that gives a lower bound for the second subderivative of the spectral function in \eqref{spec}. We should add that a similar result was recently proved in \cite[Proposition 5.3]{ms} for convex spectral functions. 
Here, we  extend this result  by removing the convexity restriction. While its   proof resembles that of \cite[Proposition 5.3]{ms}, it requires careful modifications of the argument therein.

\begin{Proposition}[Lower estimate for second subderivatives] \label{lbs1} Assume that $g:\S^n\to \oR$ is lsc and has the spectral representation in \eqref{spec} and $Y\in \sub g(X)$. 
Let $\mu_1>\cdots>\mu_r$ be the distinct eigenvalues of $X$  and $U\in \O^n(X)$.  
Then, for any $H\in \S^n$ we have 
\begin{equation}\label{lbss}
\d^2g(X,Y)(H)\ge \d^2\th\big(\lm(X),v\big)\big(\lm'(X;H)\big) + 2\sum_{m=1}^r \big\la \Diag(y)_{\al_m\al_m},U_{\al_m}^\top  H ( \mu_{m} I  - X)^{\dagger} H U_{\al_m}   \big\ra,
\end{equation}
where  $\al_m$, $m=1,\ldots,r$, are defined in \eqref{index}, and 
$y\in \sub \th(\lm(X))$ such that $Y=U\Diag(y)U^\top$ and $v=(v_{\al_1},\ldots,v_{\al_r})=Qy$ for some $Q\in  \P^n_X$ with $v_{\al_m}\in \R^{|\al_m|}_{\dn}$ for any $m=1,\ldots,r$. Moreover, for any such permutation matrix $Q$, we always have
$$
\d^2\th\big(\lm(X),v\big)\big(\lm'(X;H)\big)=\d^2\th\big(\lm(X),y\big)\big(Q^\top \lm'(X;H)\big)
$$
\end{Proposition}
\begin{proof} Since  $Y\in \sub g(X)$,  it follows from Proposition~\ref{subsp} that there are $U\in \O^n(X)$ and $y\in \sub  \th(\lm(X))$ such that $Y=U\Diag(y)U^\top$.  
Since $y=(y_{\al_1}, \ldots, y_{\al_r})$, we can choose a permutation matrix $Q_m\in \S^{|\al_m|}$ such that $v_{\al_m}=Q_m y_{\al_m}$ with  $v_{\al_m}\in \R^{|\al_m|}_{\dn}$ for any $m=1,\ldots,r$.
Set $Q= \Diag(Q_1,\ldots,Q_r)$ and observe that 
\begin{equation}\label{subg}
v=Qy\in Q\sub \th(\lm(X))= \sub \th(Q\lm(X))= \sub \th(\lm(X)).
\end{equation}
Let  $H\in \S^n$ and   pick sequences $H_k\to H$ and $t_k\dn 0$. Setting $\Delta_{t_k}\lm(X)(H_k):=\big(\lm(X+t_kH_k)-\lm(X)\big)/t_k$, we get 
\begin{eqnarray*}
\Delta^2_{t_k} g\big(X,Y\big)(H_k)&=& \frac{\th\big(\lm(X+t_kH_k)\big)- \th\big(\lm(X)\big)-t_k\big\la Y,H_k\big\ra}{\sm t_k^2}\\
&=&  \frac{\th\big(\lm(X)+t_k\Delta_{t_k}\lm(X)(H_k)\big)- \th\big(\lm(X)\big)-t_k\big\la v,\Delta_{t_k}\lm(X)(H_k)\big\ra}{\sm t_k^2}\\
&&+ \frac{\big\la v,\Delta_{t_k}\lm(X)(H_k)\big\ra -\big\la Y,H_k\big\ra}{\sm t_k}\\
&=& \Delta^2_{t_k} \th\big(\lm(X),v\big)(\Delta_{t_k}\lm(X)(H_k)\big)+ \frac{\big\la v,\Delta_{t_k}\lm(X)(H_k)\big\ra -\big\la Y,H_k\big\ra}{\sm t_k}.
\end{eqnarray*}
It follows from   $Y=U\Diag(y) U^\top$ that
\begin{equation}\label{sdcy}
\big\la Y,H_k\big\ra= \big\la U \Diag(y) U^\top,H_k\big\ra=\big\la \Diag(y),U^\top H_kU\big\ra=  \sum_{m=1}^r \big\la \Diag(y)_{\al_m\al_m},U_{\al_m}^\top  H_kU_{\al_m}\big\ra.
\end{equation}
On the other hand, it results from  \eqref{secondexp1} and     Fan's inequality \cite{Fan49} that 
\begin{eqnarray*}
\big\la v,\Delta_{t_k}\lm(X)(H_k)\big\ra &=& \sum_{m=1}^r\sum_{j\in \al_m} \frac{v_j \big( \lm_j(X+t_kH_k)-\lm_j(X)\big)}{t_k}\\
&=&\sum_{m=1}^r\sum_{j\in \al_m} v_j \lm_{\ell_j}\big(U_{\al_m}^\top H_kU_{\al_m}+t_kU_{\al_m}^\top  H_k( \mu_m I-X)^\dagger H_kU_{\al_m}\big) + O(t_k^2)\\
&\ge & \sum_{m=1}^r\big\la \Diag(y)_{\al_m\al_m}, U_{\al_m}^\top H_kU_{\al_m}+t_kU_{\al_m}^\top  H_k( \mu_m I-X)^\dagger H_kU_{\al_m}\big\ra + O(t_k^2).\\
\end{eqnarray*}
Combining this with \eqref{sdcy} brings us to 
\begin{eqnarray*}
\frac{\big\la v,\Delta_{t_k}\lm(X)(H_k)\big\ra -\big\la Y,H_k\big\ra}{\sm t_k}&\ge &  2\sum_{m=1}^r\big\la \Diag(y)_{\al_m\al_m},  U_{\al_m}^\top  H_k( \mu_m I-X)^\dagger H_kU_{\al_m}\big\ra + O(t_k).
\end{eqnarray*}
This leads us to the estimate 
\begin{eqnarray*}
\Delta^2_{t_k} g\big(X,Y\big)(H_k)&\ge & \Delta^2_{t_k} \th\big(\lm(X),v\big)\big(\Delta_{t_k}\lm(X)(H_k)\big) \\
&& + 2\sum_{m=1}^r\big\la \Diag(y)_{\al_m\al_m},  U_{\al_m}^\top  H_k( \mu_m I-X)^\dagger H_kU_{\al_m}\big\ra + O(t_k),
\end{eqnarray*}
which in turn clearly justifies the lower estimate in \eqref{lbss} for the second subderivative of $g$ at $X$ for $Y$
since $\Delta_{t_k}\lm(X)(H_k)\to \lm'(X;H)$ as $k\to \infty$. 

To prove the last claim, pick $Q\in \P^n_X$ such that 
 $v=(v_{\al_1},\ldots,v_{\al_r})=Qy$  with $v_{\al_m}\in \R^{|\al_m|}_{\dn}$ for any $m=1,\ldots,r$ and 
observe that  $Q^\top=Q^{-1}=\Diag(P_1^{-1},\ldots,P_r^{-1})\in \P^n_X$. Thus, we obtain for any $w\in \R^n$ that
\begin{eqnarray*}
\d^2 \th(\lm( X ), v) (w) &=& \liminf_{\substack{t\dn 0\\w'\to w}}\frac{\th \big(\lambda (X)+tw'\big)-\th \big(\lambda (X)\big)-t\la v,w'\ra}{\sm t^2}\\
&=& \liminf_{\substack{t\dn 0\\ w'\to w}}\frac{\th \big(\lambda (X)+tQ^\top w'\big)-\th \big(\lambda (X)\big)-t\la y,Q^\top w'\ra}{\sm t^2}\\
&\ge & \liminf_{\substack{t\dn 0\\ w'\to Q^\top w}}\frac{\th \big(\lambda (X)+tw'\big)-\th \big(\lambda (X)\big)-t\la y,w'\ra}{t}=\d^2 \th(\lm( X ), y) (Q^\top w).
\end{eqnarray*}
One can argue similarly to derive the opposite inequality $\d^2 \th(\lm( X ), v) (w)\le \d^2 \th(\lm( X ), y) (Q^\top w)$ for any $w\in \R^n$, which confirms the claimed equality.
\end{proof}

Note that when the spectral function $g$ in Proposition~\ref{lbs1} is convex, the vector $y$ therein can be chosen as $\lm(Y)$, which leads us to $v=y$, where $v$ is taken from Proposition~\ref{lbs1}. In this case, the matrix $Q$ will be the identity matrix and consequently, this result boils down to 
\cite[Proposition 5.3]{ms}. 

Our next goal is to investigate when the inequality in \eqref{lbss} becomes equality. Such a task requires knowing more about the critical cone of spectral functions. Recall that the critical of a function $f:\X\to \oR$
at $x$ for $v\in \sub f(x)$ is defined by 
$$
C_f(x,v)=\big\{w\in \X|\; \d f(x)(w)=\la w,v\ra\big\}.
$$
Note that the critical cone of a function has a close relationship with its second subderivative. To see this, taking $(x,v)\in \gph \sub f$ such that $\d^2 f(x,v)$ is proper, we can conclude from \cite[Proposition~13.5]{rw} that $\dom \d^2 f(x,v)\subset {{\cal C}_f}(x, v) $. As shown later in this section, equality requires some additional assumptions; \cite[Proposition~3.4]{ms} for more details.

The following proposition presents a characterization of  the critical cone of the spectral function, and extends \cite[Proposition 5.4]{ms2}, where a similar result was obtained for convex spectral functions.  Since its proof is similar to that of \cite[Proposition 5.4]{ms2} with similar modifications illustrated in Proposition \ref{lbs1}, we only provide a sketch of its proof.

\begin{Proposition}[critical cone of spectral functions] \label{crit} Assume that $g:\S^n\to \oR$ is lsc and has the spectral representation in \eqref{spec} and  that   $Y\in \sub g(X)$. Assume further that    one of the following conditions is satisfied:
\begin{itemize}[noitemsep,topsep=0pt]
 \item [{\rm (a)}] $g$ is  subdifferentially regular;
 \item [{\rm (b)}] $\th$  is  locally Lipschitz continuous  at $\lm(X)$ relative to its domain and $Y\in \Hat\sub g(X)$. 
 \end{itemize} 
Then, we have   $H\in C_g(X,Y)$ if and only if 
  $\lm'(X;H)\in C_\theta\big(\lm(X),v\big)$ and the matrices $\Diag(y)_{\al_m\al_m}$ and $U_{\al_m}^\top  HU_{\al_m}$ have a simultaneous ordered spectral decomposition for any $m=1,\ldots,r$, where  $\al_m$ taken from \eqref{index}
  and $U\in \O^n(X)$,  where $Y=U\Diag(y)U^\top$ for some $y\in \sub \th(\lm(X))$, and where $v$ is taken from Proposition~{\rm \ref{lbs1}}.
\end{Proposition}
\begin{proof} Suppose that either (a) or (b) holds.
Since $Y\in \sub g(X)$, it follows from Proposition~\ref{subsp} that there is $U\in \O^n(X)$ and $y \in \sub \th \big(\lm(X)\big)$ such that  $Y=U\Diag(y)U^\top$. Take the matrix $Q$ from Proposition~\ref{lbs1} and observe for any $H\in \S^n$ that 
\begin{eqnarray*}
 \big\la Y,H\big\ra &=&  \big\la \Diag(y),U^\top HU\big\ra=  \sum_{m=1}^r \big\la \Diag(y)_{\al_m\al_m},U_{\al_m}^\top  HU_{\al_m}\big\ra\nonumber\\
&\le &  \sum_{m=1}^r \big\la v_{\al_m}, \lm(U^\top_{\al_m} H U_{\al_m})\big\ra =\big\la v,\lm'(X;H)\big\ra
\le  \d \th\big(\lm(X)\big)\big(\lm'(X;H)\big)=\d g(X)(H),
\end{eqnarray*}
where the last equality results from \cite[Theorem~3.5]{ms} when (b) holds and from Theorem~\ref{specsubd} when (a) is satisfied and where the last inequality follows from $v\in \Hat \sub \th \big(\lm(X)\big)$ and \cite[Exercise~8.4]{rw} in both cases. The rest of the proof can be done as that of \cite[Proposition 5.4]{ms}, which was argued for convex spectral functions.
\end{proof}

The next result consists of some useful observations that will be utilized in our characterization of twice epi-differetiability of spectral functions.  

\begin{Lemma}\label{tpd_g}
    Assume that $g:\S^n\to \oR$ is lsc and has the spectral representation in \eqref{spec}. If $(X,Y)\in \gph  \sub g$
and $Y=U\Diag(y)U^\top$ for some $y\in  \sub\th(\lm(X))$ and $U\in \O^n(X)$, then the following properties hold.
\begin{itemize}[noitemsep,topsep=0pt]
\item [{\rm (a)}] Twice epi-differentiability of $\th$ at $\lm(X)$ for $y$ amounts to that  of $\th$ at $\lm(X)$ for $v$, where  $v$ is defined in Proposition~{\rm\ref{lbs1}}.

\item [{\rm (b)}] If $g$ is twice epi-differentiable at $X$ for $Y$, then it enjoys the same property at $\Lm(X)=\Diag(\lm(X))$ for $\Diag(y)$.

\item [{\rm (c)}] The inequality 
$$
\d^2g(\Lambda(X),\Diag (y))(\Diag(w))\le \d^2\theta(\lambda(X),y)(w)
$$
holds for any $w\in \R^n$.
\end{itemize}
\end{Lemma}
\begin{proof}
To prove (a), take the  block permutation matrix $Q\in \P_X^n$ from Proposition~{\rm\ref{lbs1}} such that 
$v=Qy$. Take $w\in\R^n$ and a sequence  $t_k\dn 0$, and assume that 
$\th$ is twice epi-differentiable at $\lm(X)$ for $y$. So, we find a sequence 
$w_k\to Q^\top w$ such that $\Delta^2_{t_k} \th\big(\lm(X),y\big)(w_k)\to \d^2\th\big(\lm(X),y\big)(Q^\top w)$. 
It is not hard to see that 
$$
 \Delta^2_{t_k} \th\big(\lm(X),v\big)(Qw_k)= \Delta^2_{t_k} \th\big(\lm(X),y\big)(w_k).
$$
This brings us to 
\begin{eqnarray*}
 \lim_{k\to \infty}\Delta^2_{t_k} \th\big(\lm(X),v\big)(Qw_k)&=& 
\lim_{k\to \infty}\Delta^2_{t_k} \th\big(\lm(X),y\big)(w_k)\\
&=&
\d^2\th\big(\lm(X),y\big)(Q^\top w)=\d^2\th\big(\lm(X),v\big)(w),   
\end{eqnarray*}
where the last equality was justified at the end of the proof of Proposition~\ref{lbs1}. This demonstrates that   $\th$ is twice epi-differentiable at $\lm(X)$ for $v$. The opposite implication can be justified similarly. 

To prove (b), it can be checked directly from \eqref{spec2} that 
\begin{equation}\label{n_ted}
\begin{cases}
 \d^2g(\Lambda(X),\mbox{Diag}(y))(W)=\d^2 g(X,Y)(UWU^\top)\quad  \mbox{and}\\
 \Delta_{t}^2 g(X,Y)(W)=\Delta_{t}^2 g\big(\Lambda(X), \Diag(y)\big)(U^\top WU)   
\end{cases}  
\end{equation}
for any $W\in \S^n$ and any $U\in \O^n(X)$ such that $Y=U\Diag(y)U^\top$.
Take $H\in \S^n$ and a sequence $t_k\dn 0$. 
Since $g$ is twice epi-differentiable at $X$ for $Y$, there are  sequences
$H_k\to  UHU^\top$ such that $\Delta_{t_k}^2 g(X,Y)(H_k)\to \d^2 g(X,Y)(UHU^\top)$. Set $\Hat H_k:=U^\top H_kU$ and observe that 
$\Delta_{t_k}^2 g(X,Y)(H_k)=\Delta_{t_k}^2 g\big(\Lambda(X), \Diag(y)\big)(\Hat H_k)$. This, together with \eqref{n_ted}, leads us to 
\begin{eqnarray*}
\lim_{k\to \infty}\Delta_{t_k}^2 g\big(\Lambda(X), \Diag(y)\big)(\Hat H_k)&=& 
\lim_{k\to \infty}\Delta_{t_k}^2 g(X,Y)(H_k)\\
&=&\d^2 g(X,Y)(UHU^\top)= \d^2g(\Lambda(X),\mbox{Diag}(y))(H),
\end{eqnarray*}
which justifies (b).  

Turning now to (c), we know from \eqref{spec2} that the symmetric function $\th$ can be represented as the composite function   $\theta =g\circ F$ with $F(x):=\Diag(x)$ for any $x\in \R^n$. This allows us to use the extensive
theory of variational properties of composite functions (cf. \cite{rw, ms, mms1}) to justify (c). 
Note that since $y\in  \sub\th(\lm(X))$, we deduce from \cite[Theorem~5]{l99} that $\Diag(y)\in \sub g\big(\Diag(\lm(X)\big)$. 
Using this and the fact that $\nabla F(\lm(X))^*\Diag(y)=y$, we can view $\Diag(y)$ as a Lagrange multiplier associated with $y$ for the composite representation of $\th$.
Therefore, a similar argument as 
the proof of \cite[Proposition~13.14]{rw} gives us the claimed inequality.
\end{proof}

After these preparations, we are now in a position to present our main result in this section in which a characterization of twice epi-differentiability of spectral functions will be given via the same property of their symmetric parts. We should add here that twice epi-differentiability for some classes of eigenvalue functions, including the maximum and leading eigenvalue functions, was justified first by Torki in \cite{t2}. The latter, however, did not provide a similar characterization as the one below by relating it to the same property of the symmetric part of a spectral function.

\begin{Theorem}\label{thm:ted} 
Assume that $g:\S^n\to \oR$ is lsc and has the spectral representation in \eqref{spec}  and that $\th$ is locally Lipschitz continuous  relative to its domain. If $(X,Y)\in \gph \Hat \sub g$
and $Y=U\Diag(y)U^\top$ for some $y\in \Hat \sub\th(\lm(X))$ and $U\in \O^n(X)$, 
  then $g$ is twice epi-differentiable at $X$ for $Y$ if and only if $\th$ is twice epi-differentiable at $\lm(X)$ for $y$. Moreover, if one of these equivalent conditions holds, then the second subderivative of $g$ at $X$ for $Y$ is calculated by 
\begin{equation}\label{ssub_calc}
\d^2g(X,Y)(H)= \d^2\th\big(\lm(X),v\big)\big(\lm'(X;H)\big) + 2\sum_{m=1}^r \big\la \Diag(y)_{\al_m\al_m},U_{\al_m}^\top  H ( \mu_{m} I  - X)^{\dagger} H U_{\al_m}   \big\ra,  
\end{equation}
for any $H\in \S^n$ such that $\d^2g(X,Y)(H)<\infty$, where  $v$ is defined in Proposition~{\rm\ref{lbs1}} and where  $\al_m$, $m=1,\ldots,r$, are defined in \eqref{index} and where  $\mu_1>\cdots>\mu_r$ are the distinct eigenvalues of $X$.
\end{Theorem}

\begin{proof}  Assume  first that $\th$ is twice epi-differentiable at $\lm(X)$ for $y$ and take $H\in \S^n$. 
Our goal is to demonstrate the validity of  \eqref{dedf} for the following two cases: 1) $\d^2g(X,Y)(H)=\infty$;
2) $\d^2g(X,Y)(H)<\infty$. For the first case, take a sequence $t_k\downarrow 0$ and set $H_k:=H$ for each $k$ and observe that 
$$
\infty=\d^2 g(X,Y)(H)\le \liminf_{k\to \infty} \Delta_{t_k}^2 g(X ,Y)(H)\le\limsup_{k\to \infty} \Delta_{t_k}^2 g(X ,Y)(H)\le \infty=\d^2g(X,Y)(H).
$$
This implies that $\lim_{k\to \infty} \Delta_{t_k}^2 g(X ,Y)(H)=\d^2g(X,Y)(H)$, which confirms \eqref{dedf} for this case. Turing now to  the second case, 
observe that given a sequence $t_k\downarrow 0$, it suffices to construct a sequence $H_k\to H$ such that 
$$
\lim_{k\to \infty} \Delta^2_{t_k} g\big(X,Y\big)(H_k) =  \d^2\th\big(\lm(X),v\big)\big(\lm'(X;H)\big) + 2\sum_{m=1}^r \big\la \Diag(y)_{\al_m\al_m},U_{\al_m}^\top  H ( \mu_{m} I  - X)^{\dagger} H U_{\al_m}   \big\ra,
$$
where $y$ and $v$ are defined in Proposition~\ref{lbs1}. Indeed, doing so leads us via Proposition~\ref{lbs1} to 
$$
\lim_{k\to \infty} \Delta^2_{t_k} g\big(X,Y\big)(H_k) = \d^2g(X,Y)(H),
$$
which proves \eqref{dedf} for the second case. According to Lemma~\ref{tpd_g}(a),  $\th$ is twice epi-differentiable at  at $\lm(X)$ for $v$.
Thus, we find a sequence $\{w^k \b$, converging to $\lm'(X;H)$, such that 
\begin{equation}\label{seq1}
 \Delta^2_{t_k} \th\big(\lm(X),v\big)(w^k)\to \d^2\th\big(\lm(X),v\big)\big(\lm'(X;H)\big).   
\end{equation}
Note that $w^k$ has an equivalent representation  $w^k=(w^k_{\al_1},\ldots,w^k_{\al_r})$ with $w^k_{\al_m}\in \R^{|\al_m|}$ for any $m=1,\ldots,r$. For each $k$ and each $m\in \{1,\ldots, r\}$, we can find  a permutation matrices $G_m^k$ such that $\tilde w^k_{\al_m}:=G_m^k w^k_{\al_m}\in \R^{|\al_m|}_{\dn}$. Since the entries of $G_m^k$ are either $0$ or $1$, we can assume by passing to a subsequence if necessary that 
for each $m\in \{1,\ldots, r\}$ the sequence $\{G_m^k\b$ is constant. 
For any $m\in \{1,\ldots, r\}$, set $G_m:=G_m^k$ and define the block diagonal permutation matrix $G\in \P^n_X$ by $G:=\Diag(G_1,\ldots,G_r)$.
It follows from  $w^k\to \lm'(X;H)$ that $\tilde w^k\to G\lm'(X;H)$.
Since $\tilde w^k\in \R^{|\al_1|}_{\dn}\times \ldots \times \R^{|\al_r|}_{\dn}$, we get $G\lm'(X;H)\in \R^{|\al_1|}_{\dn}\times \ldots \times \R^{|\al_r|}_{\dn}$. On the other hand, we already have
\begin{equation}\label{dir_eigen}
\lambda^{'} ( X ; H )=\big(\lambda (U_{\al_1}^{\top} H U_{\al_1}),\ldots,\lambda (U_{\al_r}^{\top} H U_{\al_r}) \big) \in \R^{|\al_1|}_{\dn}\times \ldots \times \R^{|\al_r|}_{\dn}.
\end{equation}
Combining these and remembering that $G$ is a permutation matrix allows us to conclude that  
\begin{equation}\label{eq:G}
    G\lm'(X;H)=\lm'(X;H),
\end{equation} 
which in turn implies that $\tilde w^k\to \lm'(X;H)$. Moreover, since $G\in \P^n_X$ and  $\th$
is symmetric, we obtain for any $k$ that  
\begin{equation}\label{th_lm}
\th \big(\lambda (X)+t_k w^k\big)=\th \big(\lambda (X)+t_k\tilde w^k\big).  
\end{equation}
By $\d^2 g(X,Y)(H)<\infty$ and the lower estimate in \eqref{lbss}, we get 
$\d^2\th\big(\lm(X),v\big)\big(\lm'(X;H)\big)<\infty$, which coupled with \eqref{seq1} tells us that $\lambda (X)+t_k w^k\in \dom \th$ for all $k$ sufficiently large. Combining it with \eqref{th_lm}, we get for all $k$ sufficiently large that
\begin{equation}\label{th_lm2}
\lambda (X)+t_k\tilde w^k\in \dom \th.
\end{equation}
Since $\d^2 g(X,Y)(H)<\infty$, we deduce from \cite[Proposition~13.5]{rw} that $\d g(X)(H)\le \la H,Y\ra$, which together with $Y\in \Hat\sub g(X)$ yields $H\in C_g(X,Y)$.  We thus conclude from Proposition~\ref{crit} that  the matrices $\Diag(y)_{\al_m\al_m}$ and $U_{\al_m}^\top  HU_{\al_m}$ have a simultaneous ordered spectral decomposition for any $m=1,\ldots,r$. Thus, we find for each $m=1,\ldots,r$ an orthogonal matrix $  P_m\in  \O^{|\al_m|}(\Diag(y)_{\al_m\al_m})\cap \O^{|\al_m|}(U_{\al_m}^{\top} H U_{\al_m})$ such that 
\begin{equation}\label{spc_ref}
\Diag(y)_{\al_m\al_m}=  P_m\Diag(v)_{\al_m\al_m}  P_m^\top\quad \mbox{and}\quad  U_{\al_m}^{\top} H U_{\al_m}=  P_m\Lambda(U_{\al_m}^\top  H U_{\al_m})  P_m^\top.
\end{equation}
Take   the  matrix $P_m$, $m=1,\ldots,r$, and define the matrix 
\begin{equation}\label{mata}
A_k:=\Diag\big(P_1\Diag(  w^k_{\al_1})P_1^\top,\ldots,P_r\Diag( w^k_{\al_r})P_r^\top\big)+\widehat{H}, \quad k\in \N, 
\end{equation}
and 
$$
D:=\Diag\big(U_{\al_1}^\top D_1U_{\al_1},\ldots,U_{\al_r}^\top D_rU_{\al_r}\big)
$$
where 
$$\widehat{H}_{ij}=\begin{cases}
0 & \mbox{if}\;i,j\in\alpha_m\;\mbox{for any} \; m=1,\ldots,r,\\
\widetilde{H}_{ij} & \mbox{otherwise},
\end{cases}$$
with $\widetilde{H}:=U^\top HU$ and  
$D_m:=  H (\mu_{m}I-X)^{\dagger} H $ for any $m=1,\ldots,r$, and then set 
\begin{equation}\label{what}
E_k:=    U(A_k- t_kD)U^\top, \quad k\in \N.
\end{equation}
Using the second identity in \eqref{spc_ref} and \eqref{dir_eigen}, one can easily see 
for each $m=1,\ldots,r$ that 
$$
P_m\Diag(  w^k_{\al_m})P_m^\top \to P_m\Diag(\lambda (U_{\al_m}^{\top} H U_{\al_m}))P_m^\top =  P_m\Lambda(U_{\al_m}^\top  H U_{\al_m})  P_m^\top=U_{\al_m}^{\top} H U_{\al_m},
$$
which in turn results in $E_k\to U^\top\widetilde{H}U=H$.
Since  $\langle Y,U\widehat{H}U^T\rangle=\langle \Diag(y),\widehat{H}\rangle=0$, 
the identity $ U_{\al_m}^\top U A_kU^\top U_{\al_m}=P_m\Diag(w^k_{\al_m})P_m^\top$, $m=1,\ldots,r$, brings us to 
\begin{equation*} 
\begin{aligned}
&\big\la Y,UA_kU^\top\big\ra = 
\sum_{m=1}^{r}  \big\la \Diag(y)_{\al_m\al_m},  U_{\al_m}^\top \big( UA_kU^\top\big) U_{\al_m}\big\ra+\langle Y,U\widehat{H}U^T\rangle\\
&= \sum_{m=1}^{r}  \big\la \Diag(y)_{\al_m\al_m},    P_m\Diag(w^k_{\al_m})  P_m^\top \big\ra
= \sum_{m=1}^{r}  \big\la   P_m^\top \Diag(y)_{\al_m\al_m}    P_m, \Diag(w^k_{\al_m}) \big\ra\\
&= \sum_{m=1}^{r}  \big\la  \Diag(v_{\al_m})   , \Diag(w^k_{\al_m}) \big\ra=\la v,w^k\ra,
\end{aligned}
\end{equation*}
where the forth equality results from the   first relationship in \eqref{spc_ref}. Thus, we obtain 
\begin{eqnarray} 
\frac{\la v,w^k\ra - \big\la Y,E_k\big\ra}{\sm t_k} &=& \frac{\la v,w^k\ra - \big\la Y,UA_kU^\top\big\ra}{\sm t_k}+2 \big\la U\Diag(y)U^\top,UDU^T\big\ra
\nonumber\\
&=&2\sum_{m=1}^{r}  \big\la \Diag(y)_{\al_m\al_m},  U_{\al_m}^\top \big( UD_mU^T\big) U_{\al_m}\big\ra \nonumber\\
&=& 2\sum_{m=1}^r \big\la \Diag(y)_{\al_m\al_m},U_{\al_m}^\top  H ( \mu_{m} I  - X)^{\dagger} H U_{\al_m}   \big\ra.\label{ted_g2}
\end{eqnarray}
Take $m\in \{1,\ldots,r\}$ and $i\in \al_m$. 
Appealing now to \eqref{secondexp1}, we can conclude that 
\begin{eqnarray*} 
&&\lambda_i (X+t_kE_k) - \lambda_i (X) \\
&=& \lm_{\ell_i}\big(t_k P_m\Diag(w^k_{\al_m})P_m^\top - t_k^2(U_{\alpha_m}^\top D_mU_{\alpha_m}-U_{\alpha_m}^\top UA_kU^\top (\mu_{m}I-X)^{\dagger} UA_kU^\top U_{\alpha_m})+o(t_k^2)\big)+ O(t_k^3).
\end{eqnarray*}
As argued above, we have $UA_kU^\top\to H$, which ensures that 
$$
t_k^2\big(U_{\alpha_m}^\top D_mU_{\alpha_m}-U_{\alpha_m}^\top UA_kU^\top (\mu_{m}I-X)^{\dagger} UA_kU^\top U_{\alpha_m}\big) =o(t_k^2).
$$ 
Thus, we arrive at
\begin{equation*} 
\lambda_i (X+t_kE_k) - \lambda_i (X) = \lm_{\ell_i}\big(t_k P_m\Diag(w^k_{\al_m})P_m^\top  +o(t_k^2)\big) + O(t_k^3).
\end{equation*}
This, coupled with the Lipschitz continuity of eigenvalues, 
leads us for each $m\in \{1,\ldots,r\}$ to 
\begin{eqnarray*}
\lambda_{\al_m} (X+t_kE_k) - \lambda_{\al_m} (X) -t_k \tilde w^k_{\al_m}& = & t_k\lm(P_m\Diag(w^k_{\al_m})P_m^\top)+ o(t_k^2) -t_k \tilde w^k_{\al_m}\nonumber\\
&= & t_k \tilde w^k_{\al_m}+ o(t_k^2) -t_k \tilde w^k_{\al_m}=o(t_k^2).
\end{eqnarray*}
Thus, we arrive at 
\begin{equation} \label{taylor_1}
\lambda (X+t_kE_k) - \lambda (X) -t_k \tilde w^k=o(t_k^2).
\end{equation}
Note that we can not guarantee that $X+t_kE_k\in \dom g$.
To resolve this issue, we can utilize \cite[Proposition~2.3]{dlms} 
to obtain for any $k$ that  
\begin{equation*}
\dist ( X+t_kE_k, \dom g) = \dist ( \lambda(X+t_kE_k),\dom \theta ),
\end{equation*}
which together with \eqref{taylor_1}  and \eqref{th_lm2} brings us to the relationships
\begin{eqnarray}
{\dist}\Big(E_k, \frac{\dom g - X}{t_k}\Big)&= &\frac{1}{t_k}\, \dist \big( \lambda(X+t_kE_k), \dom\th \big)\nonumber \\
&\le &\frac{1}{t_k}\,\big \| \lambda (X) +t_k \tilde w^k+o(t_k^2) - \lambda (X) -t_k \tilde w^k\big\|\nonumber\\
&=& o(t_k)\;\mbox{ for all }\;k\in\N.\label{eq:indom}
\end{eqnarray}
Thus, for each $k$ sufficiently large, we can find a matrix $H_k\in \S^n$ such that   $X+t_kH_k\in \dom g $ and $H_k-E_k=o(t_k)$. 
Using this observation, one can argue that 
\begin{equation}\label{ted_g}
\begin{aligned}
& \Delta^2_{t_k} g\big(X,Y\big)(H_k)=  \frac{\th\big(\lm(X+t_kH_k)\big)- \th\big(\lm(X)\big)-t_k\big\la Y,H_k\big\ra}{\sm t_k^2}\\
 =& \Delta^2_{t_k} \th\big(\lm(X),v\big)( w^k)+ \frac{\la v, w^k\ra - \big\la Y,H_k\big\ra}{\sm t_k}+ \frac{\th \big(\lambda (X+t_kH_k)\big)-\th \big(\lambda (X)+t_k  w^k\big)}{\sm t_k^2}\\
=& \Delta^2_{t_k} \th\big(\lm(X),v\big)(w^k)+ \frac{\la v, w^k\ra - \big\la Y,E_k\big\ra}{\sm t_k}+ \frac{o(t_k)}{t_k}+\frac{\th \big(\lambda (X+t_kH_k)\big)-\th \big(\lambda (X)+t_k\tilde w^k\big)}{\sm t_k^2},
\end{aligned}
\end{equation}
where the last equality results from \eqref{th_lm}. Using the fact that 
$\th$ is locally Lipshcitz continuous at $\lm(X)$ with respect to its domain,
we find $\ell\ge 0$ such that for all $k$ sufficiently large we have 
\begin{eqnarray}
 |\th \big(\lambda (X+t_kH_k)\big)-\th \big(\lambda (X)+t_k\tilde w^k\big)|&\le & \ell\, \|  \lambda (X+t_kH_k) - \lambda (X)-t_k\tilde w^k\|\nonumber\\
 &\le & \ell\, \|  \lambda (X+t_kE_k) - \lambda (X)-t_k\tilde w^k\| +o(t_k^2)\nonumber\\
 &=& o(t_k^2),\label{eq:lipth}
\end{eqnarray}
where the last inequality results from $H_k-E_k=o(t_k)$ and where the last equality comes from \eqref{taylor_1}. Passing to the limit in \eqref{ted_g} and using \eqref{seq1}  and \eqref{ted_g2} lead us to 
\begin{equation}\label{eq:kinfy}
\Delta^2_{t_k} g\big(X,Y\big)(H_k) \to \d^2\th\big(\lm(X),v\big)\big(\lm'(X;H)\big)+2\sum_{m=1}^r \big\la \Diag(y)_{\al_m\al_m},U_{\al_m}^\top  H ( \mu_{m} I  - X)^{\dagger} H U_{\al_m}   \big\ra,
\end{equation}
which ensures \eqref{dedf} for the second case and consequently justifies 
twice epi-differentiability of $g$ at $X$ for $Y$. Moreover, it confirms the claimed formula for the second subderivative of $g$ at $X$ for $Y$.

Turning now to the proof of the opposite implication, assume that $g$ is twice epi-differentiable at $X$ for $Y$.   We again consider the following two cases: 1) $\d^2\theta(\lambda(X),y)(h)=\infty$;
2) $\d^2\theta(\lambda(X),y)(h)<\infty$. For the first case, its verification can use a similar argument as the proof of the first case in the previous implication.
To ensure twice epi-differentiability of $\th$ at $\lm(X)$ for $y$ in the second case, observe first via Lemma~\ref{tpd_g}(b) that 
$g$ is twice epi-differentiable at $\Lm(X)$ for $\Diag(y)$. 
Pick any $h\in\mathbb{R}^n$ and denote $H=\Diag(h)\in\S^n$. 
It follows from twice epi-differentiability of $g$ at $\Lm(X)$ for $\Diag(y)$ that for any $t_k\downarrow0$, there exists a sequence $\{H_k\}_{k\in\N}$, converging to $H$, such that 
\begin{equation*}
\Delta_{t_k}^2g(\Lambda(X),\Diag(y))(H_k)\rightarrow \d^2g(\Lambda(X),\Diag(y))(H).
\end{equation*}
Suppose  $h^{\downarrow}=(\tilde{h}_{\al_1},\ldots,\tilde{h}_{\al_r})=Qh$ with $\tilde h_{\al_m}:=Q_m h_{\al_m}\in \R^{|\al_m|}_{\dn}$, $Q=\Diag(Q_1,\dots,Q_r)$, $Q_m\in\P^{|\alpha_m|}$. Let $V^k=\Diag(V^k_1,\dots,V^k_r)$ with $V^k_m\in \O^{|\alpha_m|}((H_k)_{\alpha_m\alpha_m})$. 
It can be checked directly that
\begin{equation*}
\begin{aligned}
&g(\Lambda(X)+t_kH_k)=g(\Lambda(X)+t_k\Diag((H_k)_{\alpha_1\alpha_1},\dots,(H_k)_{\alpha_r\alpha_r})+t_k\mbox{OD}(H_k))\\
&=g(QV^k\Lambda(X)(V^k)^\top Q^\top+t_kQ\Diag(\Lambda((H_k)_{\alpha_1\alpha_1}),\dots,\Lambda((H_k)_{\alpha_r\alpha_r}))Q^\top +t_kQV^k\mbox{OD}(H_k)(V^k)^\top Q^\top)\\
&=g(\Lambda(X)+t_kQ\Diag(\Lambda((H_k)_{\alpha_1\alpha_1}),\dots,\Lambda((H_k)_{\alpha_r\alpha_r}))Q^\top +t_kQV^k\mbox{OD}(H_k)(V^k)^\top Q^\top),
\end{aligned}
\end{equation*}
where $\mbox{OD}(H_k):= H_k-\Diag((H_k)_{\alpha_1\alpha_1},\dots,(H_k)_{\alpha_r\alpha_r})$ denotes the off block diagonal part of $H_k$ and where the last equality comes from the block diagonal structures of $\Lambda(X)$, $Q$, and $V_k$. 
Let $$\widehat{H}_k:=Q\Diag(\Lambda((H_k)_{\alpha_1\alpha_1}),\dots,\Lambda((H_k)_{\alpha_r\alpha_r}))Q^\top+QV^k\mbox{OD}(H_k)(V^k)^\top Q^\top.$$ It can be checked directly that 
\begin{equation}\label{eq:p4}
\d^2g(\Lambda(X),\Diag(y))(H)=\lim_{k\rightarrow\infty}\Delta_{t_k}^2g(\Lambda(X),\Diag(y))(\widehat{H}_k)
=\lim_{k\rightarrow\infty}\Delta_{t_k}^2g(\Lambda(X),\Diag(y))(H_k). 
\end{equation}

  Let $h_k$ be such that $(h_k)_{\alpha_m}=Q_m\lambda\big((\widehat{H}_k)_{\alpha_m\alpha_m}+t_k(\widehat{H}_k)_{\alpha_m}^\top (\mu_mI-\Lambda(X))^{\dagger}(\widehat{H}_k)_{\alpha_m}\big)$ for any $m\in \{1,\ldots,r\}$. 
Since $\widehat{H}_k\rightarrow\mbox{Diag}\,(h)$, we know  for all $m=1,\dots,r$ that 
\begin{equation}\label{eq:15}
    (\widehat{H}_k)_{\alpha_m}^\top (\mu_mI-\Lambda(X))^{\dagger}(\widehat{H}_k)_{\alpha_m}\rightarrow0
\end{equation}
as $k\rightarrow\infty$, which implies that $h_k\rightarrow h$. 
It follows from the definition of $h_k$ that for $k$ sufficiently large, 
\begin{equation*}
\begin{aligned}
&\lambda(\Lambda(X)+t_k\widehat{H}_k)-\lambda(\Lambda(X)+t_k\Diag (h_k))=\lambda(\Lambda(X)+t_k\widehat{H}_k)-\lambda(\Lambda(X))-t_kQh_k. 
\end{aligned}
\end{equation*}
Moreover,  using \eqref{secondexp1} again brings us to 
\begin{equation*}
    \begin{aligned}
&\lambda_{\alpha_m}(\Lambda(X)+t_k\widehat{H}_k)-\lambda_{\alpha_m}(\Lambda(X))=\lambda_{\alpha_m}(t_k(\widehat{H}_k)_{\alpha_m\alpha_m}+t_k^2(\widehat{H}_k)_{\alpha_m}^T(\mu_mI-\Lambda(X))^{\dagger}(\widehat{H}_k)_{\alpha_m})+o(t_k^2).\end{aligned}
\end{equation*}
Combining the above two equalities, we arrive at 
\begin{equation}\label{eq:p2}\lambda(\Lambda(X)+t_k\widehat{H}_k)-\lambda(\Lambda(X)+t_k\Diag (h_k))=o(t_k^2).\end{equation}
Since $h\in \dom \d^2\theta(\lambda(X),y)$, we can use Lemma~\ref{tpd_g}(c) and  \eqref{eq:p4} to ensure   that $\Lambda(X)+t_k\widehat{H}_k\in\dom g$ for $k$ sufficiently large. 
Consequently, using \cite[Proposition~2.3]{dlms}, \eqref{eq:p2}, and  $\th$ being  symmetric, we get for any $k$ that 
$$
\begin{aligned}0=&\,\dist(\Lambda(X)+t_k\widehat{H}_k,\dom g)=\dist(\lambda(\Lambda(X)+t_k\widehat{H}_k)),\dom \theta)\\
=&\,\dist(\lambda(\Lambda(X)+t_k\Diag(h_k)),\dom \theta)+o(t_k^2)=\dist(\lambda(X)+t_kh_k,\dom \theta)+o(t_k^2). \end{aligned}$$
Thus, for each $k$ sufficiently large, we can find  $l_k\in \R^n$ such that   $\lambda(X)+t_kl_k\in \dom \theta $ and $h_k-l_k=o(t_k)$. 
Simple calculation then leads us to 
\begin{equation}\label{eq:qtheta}
\begin{aligned}
&\Delta_{t_k}^2\theta(\lambda(X),y)(l_k)=\frac{\theta(\lambda(X)+t_kl_k)-\theta(\lambda(X))-t_k\langle y,l_k\rangle}{\frac{1}{2}t_k^2}\\
=&\frac{g(\Lambda(X)+t_k\Diag (l_k))-g(\Lambda(X))-t_k\langle \Diag(y),\Diag(h_k)\rangle}{\frac{1}{2}t_k^2}+\frac{o(t_k)}{t_k}\\ 
=&\frac{g(\Lambda(X)+t_k\widehat{H}_k)-\theta(\lambda(X))-t_k\langle \Diag(y),\widehat{H}_k\rangle}{\frac{1}{2}t_k^2}+\frac{g(\Lambda(X)+t_k\Diag (l_k))-g(\Lambda(X)+t_k\widehat{H}_k)}{\frac{1}{2}t_k^2}\\
&+\frac{\langle \Diag(y),\widehat{H}_k\rangle-\langle y,h_k)\rangle}{\frac{1}{2}t_k}+o(1).
\end{aligned}
\end{equation}
Since $\theta$ is locally Lipschitz continuous
at $\lambda(X)$ with respect to its domain, we obtain  
\begin{equation*}\label{eq:p1}
\begin{aligned}
|g(\Lambda(X)+t_k\Diag (h_k))-g(\Lambda(X)+t_k\widehat{H}_k)|\leq l\|\lambda(\Lambda(X)+t_k\widehat{H}_k)-\lambda(\Lambda(X)+t_k\Diag (h_k))\|,
\end{aligned}
\end{equation*}
which together with \eqref{eq:p2} yeilds  
\begin{equation}\label{eq:p3}
g(\Lambda(X)+t_k\Diag (h_k))-g(\Lambda(X)+t_k\widehat{H}_k)=o(t_k^2).
\end{equation}
Moreover, we   have 
\begin{equation}\label{eq:p5}
\begin{aligned}
&\frac{\langle \Diag(y),\widehat{H}_k\rangle-\langle y,h_k\rangle}{\frac{1}{2}t_k}=\frac{\langle Qy,\diag(\lambda((H_k)_{\alpha_1\alpha_1}),\dots,\lambda((H_k)_{\alpha_r\alpha_r}))\rangle-\langle Qy,Qh_k\rangle}{\frac{1}{2}t_k}
\end{aligned}
\end{equation}
 Combining \eqref{eq:15} with the definition of $h_k$, we can conclude  for any  sufficiently large $k$ that 
\begin{equation*}\label{eq:p6}\begin{aligned}
&\|\lambda((H_k)_{\alpha_m\alpha_m})-\lambda\big((\widehat{H}_k)_{\alpha_m\alpha_m}+t_k(\widehat{H}_k)_{\alpha_m}^\top (\mu_mI-\Lambda(X))^{\dagger}(\widehat{H}_k)_{\alpha_m}\big)\|\\
&=\|\lambda((\widehat{H}_k)_{\alpha_m\alpha_m})-\lambda\big((\widehat{H}_k)_{\alpha_m\alpha_m}+t_k(\widehat{H}_k)_{\alpha_m}^\top (\mu_mI-\Lambda(X))^{\dagger}(\widehat{H}_k)_{\alpha_m}\big)\|\\
&\leq \|t_k(\widehat{H}_k)_{\alpha_m}^\top (\mu_mI-\Lambda(X))^{\dagger}(\widehat{H}_k)_{\alpha_m}\|=o(t_k), 
\end{aligned}\end{equation*}
which, coupled with \eqref{eq:p5}, tell us that  
\begin{equation*}\label{eq:p7}
    \frac{\langle \Diag(y),\widehat{H}_k\rangle-\langle y,h_k\rangle}{\frac{1}{2}t_k}=o(1). 
\end{equation*}
Combining this with \eqref{eq:qtheta}, \eqref{eq:p4}, and \eqref{eq:p3}  implies that 
$$\Delta_{t_k}^2\theta(\lambda(X),y)(l_k)\rightarrow \d^2g(\Lambda(X),\Diag(y))(\Diag(h)),$$
which together with   \eqref{lbss} confirms that  $\theta$ is twice epi-differentiable at $\lambda(X)$ for $y$ and hence completes the proof.
\end{proof}

Theorem~\ref{thm:ted} provides a practical characterization of twice epi-differentiability of spectral functions and can be viewed as a far-reaching extension of \cite[Theorem~3.1]{t2} in which this property was justified for some classes of eigenvalue functions. Note that Lipschitz continuity of $\th$ with respect to its domain, assumed in Theorem~\ref{thm:ted}, doesn't impose a major restriction in applications and often  holds  automatically.

\begin{Remark}\label{tep_dd}
{\rm Note   that a characterization of directionally differentiability of the proximal mapping of spectral functions can be found in \cite[Theorem~3]{dst2}. It is well-known  that if the spectral function $g$ in \eqref{spec} is lsc and convex, its proximal mapping can be calculated by
\begin{equation}\label{prox_con}
    \prox_g(X):=\mbox{argmin}_{W\in \S^n}\big\{g(W)+\sm \|W-X\|^2\big\}=U\Diag\big(\prox_\th(\lm(X))\big)U^\top,
\end{equation}
where $U\in \O^n(X)$. Given an lsc and convex spectral function $g$ with representation \eqref{spec}, we infer   from \cite[Theorem~3]{dst2}   that $\prox_g$ is directionally differentiable at $X\in \S^n$ if and only if $\prox_\th$ is directionally differentiable at $\lm(X)$.
It also follows from \cite[Exercise~13.45]{rw} that twice epi-differentiability of $g$ at $X$ for $Y\in \sub g(X)$ amounts to directional differentiability of  $\prox_g$ at $X+Y$.
One may wonder whether the combination of these two results can give us the equivalence established in Theorem~\ref{thm:ted} for convex spectral functions. To explore such a possibility, take  $(X,Y)\in \gph \sub g$
and assume that $g$ is twice epi-differentiable at $X$ for $Y$. According to 
  \cite[Exercise~13.45]{rw}, the proximal mapping $\prox_g$
is directionally differentiable at $X+Y$. Employing now \cite[Theorem~3]{dst2}
implies that $\prox_\th$ is directionally differentiable at $\lm(X+Y)$. Appealing again to  \cite[Exercise~13.45]{rw} indicates that $\th$ is twice epi-differentiable at $z$ for $\lm(X+Y)-z$ with $z=\prox_\th(\lm(X+Y))$. It is not hard to see from \eqref{prox_con} and the spectral representation of $X$ that $z=\lm(X)$. On the other hand, since $Y\in\partial g(X)$ and $g$ is convex,   there exists $U\in\O^n$ such that
$X=U\Lambda(X)U^\top$ and $Y=U\Lambda(Y)U^\top$, which indicate that 
$$X+Y=U\Lambda(X)U^\top+U\Lambda(Y)U^\top=U(\Lambda(X)+\Lambda(Y))U^\top. $$
Thus we have $\lambda(X+Y)=\lambda(X)+\lambda(Y)$. Combining these tells us that $\th$ is twice epi-differentiable at $\lm(X)$ for $\lm(Y)$. A similar argument can prove the opposite implication, which demonstrates that for convex spectral functions, Theorem~\ref{thm:ted} can be concluded from  \cite[Theorem~3]{dst2}.  
}
\end{Remark}

The formula for the second subderivative, established in Theorem~\ref{thm:ted}, deals with the domain of this object and so one may wonder what can be said about this set. Under the hypotheses of Theorem~\ref{thm:ted}, we can conclude from   \cite[Proposition~13.5]{rw} that if $\d^2 g(X,Y)$ is a proper function, then we always have the inclusion 
\begin{equation}\label{dom_ssub}
\dom \d^2 g(X,Y):=\big\{H\in \S^n|\;\d^2 g(X,Y)(H)<\infty\big\}  \subset C_g(X,Y).
\end{equation}
Note that   equality happens to be valid in this inclusion in many applications. To elaborate more on this, recall that for a function $f:\X\to \oR$, and $\ox\in \X$ with $f(\ox)$ finite and   $w\in \X$ with $\d f(\ox)(w)$ finite,  
the parabolic subderivative (see \cite[Definition~13.59]{rw}) of $f$ at $\ox$ for $w$ with respect to $z$  is defined by 
\begin{equation*}\label{lk02}
\d^2 f(\bar x)(w\verl z)= \liminf_{\substack{
   t\searrow 0 \\
  z'\to z
  }} \dfrac{f(\ox+tw+\frac{1}{2}t^2 z')-f(\ox)-t\d f(\ox)(w)}{\frac{1}{2}t^2}.
\end{equation*}
It was shown in  \cite[Proposition~3.4]{ms} that if $Y$ is a proximal subgradient in the sense of \cite[Definition~8.45]{rw} and if $\dom \d^2 g(X)(H\verl \cdot)\neq \emptyset$ for all $H\in C_g(X,Y)$,  the inclusion in \eqref{dom_ssub} becomes equality. One may ask when the condition $\dom \d^2 g(X)(H\verl \cdot)\neq \emptyset$ holds. Indeed, this condition does hold when the function $g$ is parabolically epi-differentiable at $X$
for any $H$ in the sense of \cite[Definition~13.59]{rw}. The latter condition was extensively studied recently in \cite[Theorem~4.7]{ms2} for spectral functions. Roughly speaking, it was shown that the latter property holds for the spectral function $g$ in \eqref{spec} whenever the symmetric function $\th$ in \eqref{spec} enjoys this property. The latter automatically holds for many symmetric functions that appear in applications. We record below a result that provides sufficient conditions for ensuring equality in \eqref{dom_ssub}.
\begin{Proposition}\label{prop:dom_ssub}
Suppose that $g$ has the spectral representation in \eqref{spec} and 
$Y$ is a proximal subgradient of $g$ at $X$ and that $\th$ is locally Lipschitz continuous with respect to its domain. If $\th$ is parabolically epi-differentiable at $\lm(X)$ for $\lm'(X;H)$ for any $H\in T_{\ss\dom g}(X)$, then we have $\dom \d^2 g(X,Y)=C_g(X,Y)$. 
\end{Proposition}
 \begin{proof}
  It follows from \cite[Theorem~4.7(a)]{ms2} that $g$ is parabolically epi-differentiable at $X$ for any $H\in T_{\ss\dom g}(X)$. Since $C_g(X,Y)\subset T_{\ss\dom g}(X)$, the discussion above gives us the claimed equality via \cite[Proposition~3.4]{ms}. 
 \end{proof}

\section{Applications of Twice Epi-Differentiability}\label{sec:aptt}
This section is devoted to establishing several second-order variational properties of spectral functions as an immediate consequence of  Theorem \ref{thm:ted}. We begin by leveraging the latter result to obtain a characterization of a generalized twice differentiability of functions, recently introduced in \cite[Definition~4.1]{r25}. To present its definition, recall that a function $\ph:\X\to \oR$ is called a generalized quadratic form (GQF) on $\X$ if $\dom \ph$ is a linear subspace of $\X$ and there exists a linear symmetric operator $L$ {\rm(}i.e. $\la Lx,y\ra=\la x,Ly\ra$
for any $x,y\in \dom \ph${\rm)} from $\dom \ph$ to $\X$ such that $f$ has a representation of form 
$$
\ph(x)=\la Lx,x\ra \quad \mbox{for all}\;\; x\in \dom \ph.
$$ 
A function $f:\X\to \oR$ is called generalized twice differentiable at $x$ for $v\in \sub f(x)$ if the second subderivative $\d^2 f(x,v)$ is a GQF and twice epi-differentiable at $x$
for $v$.

\begin{Theorem}\label{gted}
    Suppose that  all the   hypothesis of Proposition~{\rm\ref{prop:dom_ssub}}  hold and that $g$
    is subdifferentially regular and $Y=U\Diag(y)U^\top$ for some $y\in \Hat \sub\th(\lm(X))$ and $U\in \O^n(X)$. Then $g$ is generalized twice differentiable at $X$ for $Y$ if and only if $\th$ has the same property at $\lm(X)$
    for $y$. 
\end{Theorem}
\begin{proof}
    Assume first that $\th$ is generalized twice differentiable at $\lm(X)$
    for $y$. By Theorem~\ref{thm:ted}, $g$ is twice epi-differentiable at $X$
    for $Y$. Moreover, since $\d^2 \th(\lm(X), y)$ is a proper function and a GQF, it follows from \eqref{ssub_calc} that  $\d^2 g(X,Y)$ is quadratic for any $H\in \dom \d^2 g(X,Y)=C_g(X,Y)$. What remains is to show that $C_g(X,Y)$ is a linear subspace of $\S^n$.
    By assumptions and \cite[Proposition~3.4]{ms}, we have $\dom \d^2 \th(\lm(X), y)=C_\th(\lm(X),y)$ is a linear subspace of $\R^n$. Since $\th$ is subdifferentially regular, we conclude from \cite[Proposition~3.3]{hs25} that $C_\th(\lm(X),y)=N_{\sub \th(\lm(X))}(y)$. 
    Thus, we infer from a well-known fact from convex analysis that 
    $y\in \ri \sub \th(\lm(X))$.  By \cite[Proposition~3.19]{ddl}, we have 
    $Y\in \ri \sub g(X)$. This, together with the identities $\dom \d^2g(X,Y)=C_g(X,Y)=N_{\sub g(X)}(Y)$, which result from Proposition~\ref{prop:dom_ssub} and \cite[Proposition~3.3]{hs25},  tells us that $C_g(X,Y)$ is a linear subspace. 

    Assume now that $g$ is generalized twice differentiable at $X$ for $Y$.
    It follows from Theorem~\ref{thm:ted} that $\th$ is twice epi-differentiable at $\lm(X)$ for $y$. It remains to show that $\d^2\theta(\lambda(X),y)$ is a GQF. To do so, we infer from
    the second part of the proof of Theorem~\ref{thm:ted} that 
    \begin{equation}\label{ssub_th}
        \d^2\theta(\lambda(X),y)(w)=\d^2g\big(\Diag(\lm(X)),\Diag (y)\big)(\Diag(w)) 
    \end{equation}
    for any $w\in \R^n$. On the other hand, Lemma~\ref{tpd_g}(b), tells us that $g$ is twice epi-differentiable at $\Diag(\lm(X))$ for $\Diag(y)$.
    Moreover, the proof of the latter result indicates that 
    $$
    \d^2g\big(\Diag(\lm(X)),\Diag (y)\big)(H)= \d^2g\big(X,Y)(UHU^\top)
    $$
    for any $H\in \S^n$. Combining these implies that $g$ is generalized twice differentiable at $\Diag(\lm(X))$ for $\Diag(y)$. This, along with \eqref{ssub_th}, confirms that $\th$ is generalized twice differentiable at  $\lm(X)$ for $y$ and so completes the proof.
\end{proof}

\begin{Corollary}
    Suppose that $g$ has the spectral representation in \eqref{spec} and $\th$ therein is a polyhedral function. Then, if $Y\in \ri \sub g(X)$, then $g$
    is generalized twice differentiable at $X$ for $Y$.
\end{Corollary}
\begin{proof}
    It results from \cite[Exercise~13.61]{rw} that $\th$ is parabolically epi-differentiable at $\lm(X)$ for any $y\in \sub \th(\lm(X))$. Moreover, we know from \cite[Proposition~3.19]{ddl} and $Y\in \ri \sub g(X)$ that there exists  $y\in \ri\th(\lm(X))$ such that $Y=U\Diag(y)U^\top$ and $U\in \O^n(X)$. According to \cite[Proposition~13.9]{rw}, $\th$
    is twice epi-differentiable at $\lm(X)$ for $y$ and $\d^2\theta(\lambda(X),y)=\dd_{C_\th(\lm(X),y)}$. Combining this, $y\in \ri\th(\lm(X))$,  and Corollary~\ref{gted} ensures that $g$
    is generalized twice differentiable at $X$ for $Y$.
\end{proof}

We proceed with exploring twice differentiability of spectral functions. 
Note that it was proven in \cite[Theorem~3.3]{ls} that the spectral function $g$ from \eqref{spec} is twice differentiable at $X\in \S^n$ if and only if $\th$ enjoys this property at $\lm(X)$. It is tempting to ask whether our approach  can be used to derive the latter result. We show below that this can be achieved when the spectral function $g$ is prox-regular. To this end, recall that a function $f:\R^n\to \oR$ is called semidifferentiable at $\ox$ if it is finite at $\ox$ and the ``$\liminf$" in \eqref{semi1} is a ``$\lim$."
Furthermore, $f$ is said to be twice semidifferentiable at  $\ox$ if it is semidifferentiable at $\ox$ and the ``$\liminf$" in 
\begin{equation}\label{ssud}
\d^2 f(\ox)(w):=\liminf_{\substack{t\dn 0\\
w'\to w}} \frac{f(\ox+tw')-f(\ox)-t\d f(\ox)(w')}{\sm t^2},\quad w\in \X
\end{equation}
 is a ``$\lim$."
In such a case, we refer to the left-hand side of \eqref{ssud} as the second semiderivative of $f$ at $\ox$. It is known that  the twice semidifferentiability  has serious limitations to handle nonsmoothness
in variational analysis; see   \cite[p.\ 590]{rw} for a detailed discussion. 
Below, we provide a characterization of twice semidifferentiability of differentiable spectral functions.  
\begin{Theorem}\label{prop:sed}
    Suppose that $\theta$ is locally Lipschitize continuous around $\lm(X)$,  differentiable at $\lambda(X)$. Then the spectral function $g$ from \eqref{spec}  is twice semidifferentiable at $X$ if and only if $\th$ is  twice semidifferentiable at $\lambda(X)$. If one of these equivalent conditions holds, then the second semiderivative of $g$ at $X$ can be calculated for any $H\in \S^n$ by  
    $$
    \d^2g(X)(H)=\d^2\theta(\lambda(X))(\lm'(X;H))+2\sum_{m=1}^r \big\la \Diag(y)_{\al_m\al_m},U_{\al_m}^\top  H ( \mu_{m} I  - X)^{\dagger} H U_{\al_m}   \big\ra,
    $$
    where  $\al_m$, $m=1,\ldots,r$, are defined in \eqref{index}, $U\in\O^n(X)$, and $y=\nabla \th(\lm(X))$ and where  $\mu_1>\cdots>\mu_r$ are the distinct eigenvalues of $X$.
    Moreover, $\d^2g(X)(H)=\d^2g(X,\nabla g(X))(H)$ for any $H\in \S^n$. 
\end{Theorem}
\begin{proof} Assume first that $\th$ is twice semidifferentiable at $\lambda(X)$. To prove  twice semidifferentiability of $g$ at $X$, take $H\in \S^n$ and pick  sequences  $t_k\downarrow0$ and $H_k\rightarrow H$, and observe that  
$$\begin{aligned}\Delta^2_{t_k}g(X,\nabla g(X))(H_k)&=&\frac{\th\big(\lm(X+t_kH_k)\big)- \th\big(\lm(X)\big)-t_k\big\la \nabla g(X),H_k\big\ra}{\sm t_k^2}= \Delta^2_{t_k} \th\big(\lm(X),y\big)( w^k)\\
&&+ \frac{\la y, w^k\ra - \big\la \nabla g(X),H_k\big\ra}{\sm t_k}+ \frac{\th \big(\lambda (X+t_kH_k)\big)-\th \big(\lambda (X)+t_k  w^k\big)}{\sm t_k^2},
\end{aligned}$$
where $\nabla g(X)=U^\top{\rm Diag}(y)U$, $y=\nabla\theta(\lambda(X))$, and $$(w^k)_{\alpha_m}=\lambda(U^\top_{\alpha_m}H_kU_{\alpha_m}+t_kU_{\al_m}^\top  H ( \mu_{m} I  - X)^{\dagger} H U_{\al_m})\quad \mbox{for all}\;\;m=1,\ldots,r,$$ 
with the index sets $\al_m$   taken from \eqref{index}. It can be checked directly that $w^k\rightarrow\lambda'(X,H)$. Since $\theta$ is differentiable at $\lambda(X)$, we get $\Hat \sub\th(\lm(X))=\{y\}$. Because $\theta$ is symmetric, we conclude for any permutation matrices $P$ satisfying  $\lambda(X)=P\lambda(X)$ that $Py\in \Hat \sub\th(\lm(X))$, which leads us to $Py=y$. Thus, we can assume without loss of generality that 
$y\in\R_{\downarrow}^{|\alpha_1|}\times\dots\times\R_{\downarrow}^{|\alpha_r|}$. This implies that  for any $m=1,\ldots,r$, there exists $u_m\in\R$ such that $y_{\alpha_m}=u_m{\bf 1}_{|\alpha_m|}$, where ${\bf 1}_{|\alpha_m|}$ is a vector in $\R^{|\alpha_m|}$ with all entries being $1$. Hence,  the matrices $\Diag(y_{\alpha_m})$ and $U^\top_{\alpha_m}H_kU_{\alpha_m}+t_kU_{\al_m}^\top  H ( \mu_{m} I  - X)^{\dagger} H U_{\al_m}$  are simultaneously diagonalizable for any $m=1,\ldots,r$. Appealing to Fan's inequality from \cite{Fan49} leads us to 
$$
\begin{aligned}
&~~~\frac{\la y, w^k\ra - \big\la \nabla g(X),H_k\big\ra}{\sm t_k}\\
&=\frac{\sum_{m=1}^r\la y_{\alpha_m}, \lambda(U^\top_{\alpha_m}H_kU_{\alpha_m}+t_kU_{\al_m}^\top  H ( \mu_{m} I  - X)^{\dagger} H U_{\al_m})\ra - \big\la \nabla g(X),H_k\big\ra}{\sm t_k}\\
&=\frac{\sum_{m=1}^r\la {\rm Diag}(y_{\alpha_m}), U^\top_{\alpha_m}H_kU_{\alpha_m}+t_kU_{\al_m}^\top  H ( \mu_{m} I  - X)^{\dagger} H U_{\al_m}\ra - \big\la \nabla g(X),H_k\big\ra}{\sm t_k}\\
&=\frac{\la {\rm Diag}(y),  U^\top H_kU\ra-\big\la \nabla g(X),H_k\big\ra}{\sm t_k}+2 \sum_{m=1}^r\big\la y_{\alpha_m},\lambda(U_{\al_m}^\top  H ( \mu_{m} I  - X)^{\dagger} H U_{\al_m})\big\ra
\nonumber\\
&=\frac{\la y,  \lambda'(X,H_k))\ra-\la {\rm Diag}(y),U^\top H_kU\big\ra}{\sm t_k}+2\sum_{m=1}^r \big\la \Diag(y)_{\al_m\al_m},U_{\al_m}^\top  H ( \mu_{m} I  - X)^{\dagger} H U_{\al_m}   \big\ra \nonumber\\
&= 2\sum_{m=1}^r \big\la \Diag(y)_{\al_m\al_m},U_{\al_m}^\top  H ( \mu_{m} I  - X)^{\dagger} H U_{\al_m}   \big\ra.
\end{aligned}
$$
On the other hand,   we can obtain from \eqref{secondexp1} that 
$$\lambda (X+t_kH_k)-\lambda (X)-t_k  w^k=o(t_k^2).$$
Thus, it follows  from the local Lipschitz continuity of $\theta$ around $\lm(X)$ and \cite[Theorem 1.5]{t1}, and applying a similar argument as the proof of the estimate in  \eqref{eq:lipth} that 
$$\frac{\th \big(\lambda (X+t_kH_k)\big)-\th \big(\lambda (X)+t_k  w^k\big)}{\sm t_k^2}\rightarrow0.$$
Since $\theta$ is twice semidifferentiable at $\lambda(X)$ and differentiable at $\lm(X)$, we arrive at 
\begin{equation*}
 \lim_{k\to \infty} \Delta^2_{t_k} \th\big(\lm(X),y\big)= \d^2\theta(\lambda(X))(H)+ 2\sum_{m=1}^r \big\la \Diag(y)_{\al_m\al_m},U_{\al_m}^\top  H ( \mu_{m} I  - X)^{\dagger} H U_{\al_m}   \big\ra. 
\end{equation*}
Recall that  the sequences $H_k$ and $t_k$ were taken arbitrarily. Combining these tells us that  $g$ is twice semidifferentiable at $X$ and gives us  the claimed formula for the second semiderivative of $g$ at $X$. The last claim about the equivalence  of the second subderivative and second semiderivative of $g$ follows from the fact that $g$ is differentiable and $\d g(X)(H)=\la \nabla g(X),H\ra$. 

Finally, if $g$ is twice   semidifferentiable at $X$, it follows from the composite representation of $\th$ in \eqref{spec2} and \cite[Proposition~8.2(i)]{mms2} that $\th$ is twice   semidifferentiable at $\lm(X)$.  
\end{proof}

The result above opens the door to characterize twice differentiability of spectral functions that are prox-regular.  To do so, 
we say that $f$ has a quadratic expansion at $\ox$ if $f$ is differentiable at $\ox$ and there exists a linear mapping $A:\X\to \X$ such that 
$$
f(x)=f(\ox)+\la \nabla f(\ox), x-\ox\ra+ \frac{1}{2}\la A(x-\ox),x-\ox\ra +o(\|x-\ox\|^2).
$$
Recall also that a function  $f\colon\X\to\oR$ is called prox-regular at $\ox$ for $\ov$ if $f$ is finite at $\ox$ and locally lsc  around $\ox$ with $\ov\in\sub f(\ox)$, and there exist 
constants $\ve>0$ and $r\ge 0$ such that
\begin{equation}\label{prox}
\begin{cases}
f(x')\ge f(x)+\la v,x'-x\ra-\frac{r}{2}\|x'-x\|^2\;\mbox{ for all }\; x'\in\B_{\ve}(\ox)\\
\mbox{whenever }\;(x,v)\in(\gph\sub f)\cap\B_{\ve}(\ox,\ov)\; \mbox{ with }\; f(x)<f(\ox) +\ve.
\end{cases}
\end{equation}
The function $f$ is called prox-regular at $\ox$ if it enjoys this property at $\ox$ for any $\ov\in \sub f(\ox)$. We are now in a position to provide a characterization of twice differentiability of spectral functions that are prox-regular. We should note that this was obtained in \cite[Theorem~3.3]{ls}
without assuming prox-regularity via a different approach. We, however, derive it as an immediate conclusion of our second-order variational analysis of spectral functions. It remains as an open question how to drop  prox-regularity in the following result using our approach. 

\begin{Corollary}
    Assume that $g:\S^n\to \oR$ has the spectral representation in \eqref{spec} and that  $g$ is  prox-regular at $\lm(X)$. 
    Then, $g$ is twice differentiable at $X$ if and only if $\th$ is twice differentiable at $\lm(X)$. If one of these equivalent conditions holds, the second derivative of $g$ at $X$ can be calculated for any $H\in \S^n$
    by 
    $$
    \big\la\nabla^2g(X)H,H\big\ra=\big\la \nabla^2\theta(\lambda(X))\lm'(X;H),\lm'(X;H)\big\ra+2\sum_{m=1}^r \big\la \Diag(y)_{\al_m\al_m},U_{\al_m}^\top  H ( \mu_{m} I  - X)^{\dagger} H U_{\al_m}   \big\ra,
    $$
    where  $\al_m$, $m=1,\ldots,r$, are defined in \eqref{index}, $U\in\O^n(X)$, and $y=\nabla \th(\lm(X))$ and where  $\mu_1>\cdots>\mu_r$ are the distinct eigenvalues of $X$. 
\end{Corollary}
\begin{proof}
    If $g$ is twice differentiable, then it follows from the composite representation $\theta$ in \eqref{spec2}
    that $\th$ is twice differentiable at $\lm(X)$. Observe that in this case  {prox-regularity} 
    of $g$ is not required. 

    Assume now that $\th$ is twice differentiable at $\lm(X)$. 
    This requires that $\th$ be differentiable around $\lm(X)$
    which implies that its gradient must be locally bounded around this point and thus it is locally Lipschitz continuous around this point. It results then from \cite[Theorem~1.1]{l962} that   $g$ is locally differentiable around $X$.  Moreover,  
    twice differentiability of $\th$ at $\lm(X)$ implies  twice semidifferentiablility of $\th$ at $\lm(X)$  and $\d^2\th\big(\lm(X) )\big)(w)=\la \nabla^2\th(\lm(X))w,w\ra$ for any $w\in \R^n$.  
By  Theorem~\ref{prop:sed}, $g$ is twice semidifferentiable at $X$. Appealing now to \cite[Example~13.8]{rw}   confirms that $g$ has a quadratic expansion at $X$. Since $g$ is  prox-regular at $X$,
we deduce from \cite[Corollary~13.42]{rw} that $\nabla g$ is differentiable, which amounts to saying that $g$ is twice differentiable at $X$.
\end{proof}

Our next result provides a characterization of subdifferential continuity for spectral functions. Recall that a function $f:\X\to \oR$ is called subdifferentially continuous at $\ox$ for $\ov\in \sub f(\ox)$ if the convergence $(x_k,v_k)\to(\ox,\ov)$ with $v_k\in\sub f(x_k)$ yields $f(x_k)\to f(\ox)$ as $k\to\infty$. When this property holds for all subgradients $v\in \sub f(\ox)$, we call $f$ subdifferentially continuous at $\ox$.
For spectral functions, it was shown in \cite[Theorem~4.2]{dlms} that the spectral function $g$ from \eqref{spec} is prox-regular at $X$ if and only if $\th$ enjoys this property at $\lm(X)$. A similar observation can be made for subdifferential continuity as shown below.

\begin{Proposition}Assume that $g:\S^n\to \oR$ has the spectral representation in \eqref{spec}.  Then, $g$ is subdifferentially continuous at $X$ if and only if $\theta$ is subdifferentially continuous at $\lambda(X)$. 
\end{Proposition}
\begin{proof} Suppose first that $\th$ is subdifferentially continuous at $\lambda(X)$. To show  that $g$ is subdifferentially continuous at $X$, take   $Y\in\partial g(X)$. 
If $(X_k,Y_k)\rightarrow (X,Y)$ with $Y_k\in\partial g(X_k)$, we know that there exists $y_k$ such that  $y_k\in\partial\theta(\lambda(X_k))$ and $\lambda(X_k)\rightarrow\lambda(X)$. 
It is known from \cite[Lemma 14]{dlms} that there exists a permutation matrix $P$ such that $y_k\rightarrow Py$ with $Py\in\partial \theta(\lambda(X))$. Since $\theta$ is subdifferentially continuous at $\lambda(X)$, we know that $\theta(\lambda(X_k))\rightarrow\theta(\lambda(X))$, which implies $g(X_k)\rightarrow g(X)$. Thus, $g$ is subdifferentially continuous at $X$. 
To show the opposite implication, assume that $g$ is subdifferentially continuous at $X$. We claim that $g$ is subdifferentially continuous at $\Lm(X)=\Diag(\lm(X))$. To see this, take $Y\in \sub g(\Lm(X))$ and $U\in \O^n(X)$, and assume that $(X_k,Y_k)\to (\Lm(X),  Y)$ with $Y_k\in \sub g(X_k)$. In this case, it is not hard to see that 
$$
UYU^\top\in \sub g(X)\quad \mbox{and}\quad UY_kU^\top \in  \sub g(UX_kU^\top).
$$
for any $U\in \O^n(X)$. Since $(UX_kU^\top,UY_kU^\top)\to (X,UYU^\top)$ and since $g$
is is subdifferentially continuous at $X$, we get 
$$
g(X_k)=g(UX_kU^\top)\to g(X)=g(\Lm(X)),
$$
which implies that $g$ is subdifferentially continuous at $\Lm(X)$. To show that $\th$ is  subdifferentially continuous at $\lm(X)$, take $y\in \sub \th(\lm(X))$ and $(x_k,y_k)\to (\lm(X),y)$ with $y_k\in \sub \th(x_k)$. 
It follows from \cite[Theorem~1.4]{l962} that 
$$
\Diag(y)\in\sub g(\Diag(x))\quad \mbox{and}\quad \Diag(y_k)\in \sub g(\Diag(x_k)).
$$
Since $(\Diag(x_k),\Diag(y_k))\to (\Lm(X), \Diag(y))$ and since 
$g$ is subdifferentially continuous at $\Lm(X)$, we arrive via \eqref{spec2} at 
$$
\th(x_k)=g(\Diag(x_k))\to g(\Lm(X))=\th(\lm(X)),
$$
which confirms that $\th$ is  subdifferentially continuous at $\lm(X)$ and hence ends the proof.
\end{proof}

Next, we are going to use our characterization in Theorem~\ref{thm:ted} to study proto-differentiability of subgradient mappings of spectral functions. To achieve our goal, recall that the graphical derivative of a set-valued mapping $F:\X\tto\Y$ at $\ox$ for $\oy\in F(\ox)$, denoted by $DF(\ox, \oy)$, is a set-valued mapping from $\X$ into $\Y$, defined by
 $$
 \gph DF(\ox, \oy)= T_{\ss\gph F}(\ox,\oy).
 $$ 
If, in addition, for any $\eta\in DF(\ox, \oy)(w)$ and any choice of $t_k\dn 0$,  there exist sequences $w_k \to w$ and $\eta_k \to \eta$ with $\oy+t_k\eta_k \in F(\ox +t_kw_k)$, 
 then $F$ is said to be proto-differentiable at $\ox$ for $\oy$; see  \cite[page~331]{rw} for more details. One of the main sources of proto-differentiability in variational 
 analysis is  subgradient mappings of various classes of functions.

\begin{Theorem}
Assume that $g:\S^n\to \oR$ has the spectral representation in \eqref{spec}  and that $\th$ is locally Lipschitz continuous  relative to its domain and $(X,Y)\in \gph \sub g$. Suppose that  $g$ is prox-regular and subdifferentially continuous at $X$. Then 
$\sub g$ is proto-differentiable at $X$ for $Y$ if and only if 
$\sub \th$ is proto-differentiable at $\lm(X)$ for $y$, where $y\in   \sub\th(\lm(X))$ such that $Y=U\Diag(y)U^\top$ for some   $U\in \O^n(X)$.  
\end{Theorem}
\begin{proof} Our plan is to combine  Theorem~\ref{thm:ted} 
and \cite[Theorem~13.40]{rw}. To this end, recall that since $g$ is prox-regular at $X$ for $Y$, then $Y\in \Hat\sub g(X)$. Assume that $\th$ is  proto-differentiable at $\lm(X)$ for $y$. It follows from 
 \cite[Theorem~13.40]{rw} that $\th$ is twice epi-differentiable at  $\lm(X)$ for $y$. Appealing now to  Theorem~\ref{thm:ted} tells us that $g$ is twice epi-differentiable at $X$ for $Y$, which together with 
 \cite[Theorem~13.40]{rw} confirms that $g$ is proto-differentiable at $X$ for $Y$. The opposite claim can be proven similarly using 
  Theorem~\ref{thm:ted}. 
\end{proof}

We end this section with a characterization of  directional differentiability   of the proximal mapping of spectral functions. This property shows great potential in establishing local superlinear convergence of semismooth Newton method and forward backward envelop acceleration \cite{KMP23}. Recall that proximal mapping 
of a function $f:\X\to \oR$ for a parameter value  $\gamma>0$,  denoted by  $\prox_{\gamma f}$, is  defined  by 
\begin{equation*}\label{proxmap}
\prox_{\gamma f}(x)= \argmin_{w\in \X}\Big\{f(w)+\frac{1}{2\gamma}\|w-x\|^2\Big\}.
\end{equation*}
Recall from  \cite[Exercise~1.24]{rw}  that a function $f:\X\to \oR$ is called  prox-bounded if the function $f+\al\|\cdot\|^2$ is bounded from below on $\X$ for some $\al\in \R$. 

\begin{Corollary}\label{coro:prox-d-d}
Assume that $g:\S^n\to \oR$ has the spectral representation in \eqref{spec}  and that $\th$ is locally Lipschitz continuous  relative to its domain and $(X,Y)\in \gph \sub g$. Suppose that  $g$ is prox-regular and subdifferentially continuous at $X$   and that $g$ is prox-bounded. Then the following statements are equivalent.
\begin{itemize}[noitemsep,topsep=0pt]
 \item [{\rm (a)}] For any $\gamma>0$ sufficiently small, the proximal mapping  $\prox_{\gamma g}$ is directionally differentiable at $X+\gamma Y$.
 \item [{\rm (b)}] For any $\gamma>0$ sufficiently small, the proximal mapping $\prox_{\gamma \th}$  is  directionally differentiable at $\lm(X)+\gamma y$, where 
 \end{itemize} 
\end{Corollary}
\begin{proof}
    According to \cite[Exercise~13.45]{rw}, twice epi-differentiability of $g$ at $X$ for $Y$ amounts to directional differentiability of $\prox_{\gamma g}$ at $X+\gamma Y$ for any $\gamma>0$ sufficiently small. The same observation holds for $\th$. Since twice epi-differentiability of $g$ and $\th$ are equivalent, it follows from the discussion above that (a) and (b) are equivalent. 
\end{proof}

As mentioned in Remark~\ref{tep_dd}, when $g$ is convex, Corollary~\ref{coro:prox-d-d} can be obtained from \cite[Theorem~3]{dst2}. The above result, however, goes far beyond convexity and covers a large class of spectral functions that are not necessarily convex.

\section{Illustrated Examples}\label{secExp}

In this section, we are going to apply our second-order variational theory for several important classes of spectral functions to illustrate its efficacy
 in the study of important practical optimization problems. 

\subsection{Second-Order Variational Properties of Leading Eigenvalues}

Our first goal is to explore the second-order variational properties of the leading eigenvalue functions by using theoretical results in Section \ref{sect05}.  Recall from \eqref{ell} that  the eigenvalue $\lm_{i-\ell_i(X)+1}(X)$, ranking first in the group of eigenvalues equal to $\lm_i(X)$    is called a {\em leading}
eigenvalue of $X$. To express this class of functions in the form of the spectral functions in \eqref{spec},  define the $i$-th order statistic function $\phi_i$ by $\phi_i(x):=\lambda_i(\Diag (x))= i\mbox{-th largest element of } \{x_1,\dots,x_n\}$. It is not hard to see that all the leading eigenvalue functions have the spectral representation   
$$g(X):= \phi_i(\lambda(X)) \quad \mbox{with either} \;\;i=1\;\mbox{or} \;\; \phi_{i-1}(\lambda(X))>\phi_i(\lambda(X)) 
$$
for any $X\in \S^n$. 
 The leading eigenvalue function has important applications in various fields. For example, the largest eigenvalue is usually used in the principal component analysis for dimension reduction \cite{J11}. The Fiedler value, which is defined as the second smallest eigenvalue,  measures graph connectivity and is applied for image segmentation or community detection in social networks \cite{JLZD24}. Other applications includes neural network \cite{NM,OS}, signal processing \cite{SNFOV}; 
see \cite{PW18,AGZ10,GD14, EY17} for more applications. 
Such a representation was well  
studied in \cite[Section~9]{l99} through which some  first-order variational properties of this function were explored. Our main objective is to demonstrate that our developments in previous sections allow us to characterize second-order variational properties of the leading eigenvalues. We should add here that some of these properties were studied before in \cite{t2} using a different approach and without appealing to the spectral representation of the leading eigenvalue functions. We begin by recording some of the first- and second-order properties of the $i$-th order statistic function.

\begin{Proposition}\label{stat}
    Suppose that $x\in \R^n$ and $\phi_i$ is  the $i$-th order statistic function $\phi_i$ where either $i=1$ or $\phi_{i-1}(x)>\phi_i(x)$. Then the following properties are satisfied. 
\begin{enumerate}[noitemsep,topsep=0pt]
\item [\rm(a)] For any $z=(z_1,\ldots,z_n)$ sufficiently close to $x=(x_1,\ldots,x_n)$, we have $\phi_i(z)=\max\{z_i\,|\, i\in I(x)\}$, where  $I(x)=\big\{j\in\{1,\cdots,n\}|\, x_j=\phi_i(x)\big\}$. 
\item [\rm(b)]  $\phi_i$ is locally Lipschitz continuous and subdifferentially regular at $x$. Moreover, we have 
$\d \phi_i(x)(w)=\max\{w_i\,|\, i\in I(x)\}$ for any $w=(w_1,\ldots,w_n)$ and  $\sub \phi_i(x)={\rm conv}\,\{e_i|i\in I(x)\}$. 
\item [\rm(c)] The function $\phi_i$ is twice epi-differentiable at $x$
for any $v\in \sub \phi_i(x)$. 
\item [\rm(d)] The function $\phi_i$ is prox-regualr and subdifferentially continuous at $z$ for any $v\in \sub \phi_i(z)$. 
\end{enumerate}
\end{Proposition}
\begin{proof}
The claim in (a) falls directly from the definition of $\phi_i$. The claims in Part (b) are immediate consequences of (a); see also \cite[Theorem~9]{l99}. The formula for  $\sub \phi_i(x)$ results from the equivalent local representation  of $\phi_i$ from (a) and the well-known formula for the pointwise maximum of convex functions; see also \cite[Proposition~6]{l99}. The formula for the subderivative of $\phi_i$ results from \cite[Exercise~8.31]{rw}. Part (c) follows from (a) and \cite[Example~13.10]{rw}. 
Finally, (d) falls out directly from the local representation of $\phi_i$ as a polyhedral function together with \cite[Example~13.30]{rw}. 
\end{proof}

Using the properties of the $i$-th order statistic function $\phi_i$, listed above, we first calculate the subderivative of these functions.
\begin{Corollary}
 Suppose that $X\in \S^n$ and that  $\mu_1>\cdots>\mu_r$ are the distinct eigenvalues of $X$. Then, for any $i\in \{1,\ldots,r\}$,  the subderiavative of the leading eigenvalue function $\lm_i$ can be calculated by 
\begin{equation}\label{subd_lead}
    \lm_i'(X,H)=\d \lm_i(X)(H)=\lm_1(U_{\al_i}^{\top} H U_{\al_i})
\end{equation} 
for any $H\in \S^n$, where the index sets $\al_i$ comes from \eqref{index}. Moreover, $\lm_i$ is subdifferentially regular at $X$
and its subdifferential at this point has a representation in the form 
\begin{equation}\label{subd_lead2}
    \sub\lm_i(X)=\big\{U_{\al_i}\Diag(y)_{\al_i\al_i}U_{\al_i}^\top\,\big|\, U\in \O^n(X),\;\; y\in \sub\phi_i(\lm(X))\big\}.
\end{equation}
\end{Corollary}
\begin{proof}It is easy to see that for any $i\in \{1,\ldots,r\}$, the 
    leading eigenvalue function $\lm_i$ has the spectral representation 
    as \eqref{spec} with $\th$ being the $i$-th order statistic function $\phi_i$, where either $i=1$ or $\phi_{i-1}(x)>\phi_i(x)$. To justify \eqref{subd_lead}, it follows from Theorem~\ref{specsubd} that 
 \begin{eqnarray*}
    \d \lm_i(X)(H) &= & \d \phi_i(\lm(X))(\lm'(X,H))\\
    &=& \max\{\lm_s(U_{\al_i}^{\top} H U_{\al_i})\,|\, s\in \al_i\}=
    \lm_1(U_{\al_i}^{\top} H U_{\al_i}).
 \end{eqnarray*}  
 By Proposition~\ref{subsp},   the subdifferential of $\lm_i$ at $X$ has a representation of the form 
 $$
  \sub\lm_i(X)=\big\{U\Diag(y)U^\top\,\big|\, U\in \O^n(X),\;\; y\in \sub\phi_i(\lm(X))\big\}.
 $$
 By Proposition~\ref{stat}(b), $y\in \sub\phi_i(\lm(X))$ amounts to 
 $y\in \mbox{conv}\,\{e_i\,|\,i\in \al_i\}$. Thus, we get 
 $$
  Y=U\Diag(\lm(Y))U^\top=\sum_{i=1}^n \lm_i(Y)U_iU_i^\top=\sum_{s\in \al_i}  \lm_s(Y)U_sU_s^\top=U_{\al_i}\Lm(Y)_{\al_i\al_i}U_{\al_i}^\top.
 $$
 Combining these confirms \eqref{subd_lead2}.
 Finally subdifferential regualrity of $\lm_i$ at $X$ was taken from \cite[Corollary~4]{l99} along with the fact that 
 in the spectral representation of $\lm_i$, the symmetric function $\th=\phi_i$ is subdifferentially regualr according to Proposition~\ref{stat}(b). 
\end{proof}

Next, we characterize the twice epi-differentiablity of the leading eigenvalues. It is worth mentioning that although Proposition \ref{stat}(a) shows that $\phi_i$ can be written locally as a  polyhedral function, that local representation is not symmetric. Thus, existing results in \cite{ms2} are not applicable here. 
\begin{Theorem} \label{fop_lead}Suppose that $X\in \S^n$ and that  $\mu_1>\cdots>\mu_r$ are the distinct eigenvalues of $X$. Then, for any $i\in \{1,\ldots,r\}$,  the leading eigenvalue function $\lm_i$ is twice epi-differentiable at $X$ for any $Y\in \sub \lm_i(X)$ and 
$$
\d^2\lm_i(X,Y)(H)=\dd_{C_g(X,Y)}(H)+ 2\big\la Y, H ( \mu_{i}I- X)^{\dagger} H  \big\ra\quad \mbox{for all}\;\; H\in \S^n,
$$
where $U\in \O^n(X)$ such that $Y=U\Diag(y)U^\top$ for some $y\in \sub\phi_i(\lm(X))$. 
\end{Theorem}
\begin{proof} It is easy to see that for any $i\in \{1,\ldots,r\}$, the 
    leading eigenvalue function $\lm_i$ has the spectral representation 
    as \eqref{spec} with $\th$ being the $i$-th order statistic function $\phi_i$,  where either $i=1$ or $\phi_{i-1}(x)>\phi_i(x)$. It follows  Proposition~\ref{stat}(b) that $\lm_i$ is locally Lipscitz continuous and subdifferentially regular at $X$.
    Appealing now to Proposition~\ref{stat}(c) and Theorem \ref{thm:ted} confirms that $\lm_i$ is twice epi-differentiable at $X$ for $Y$.
    To justify the claimed formula for the second subderivative of $\lm_i$, we proceed with two Claims.
    
\noindent{\it Claim 1:} We have $\dom \d^2\lm_i(X,Y)=C_{\lm_i}(X,Y)$.

To justify this claim, we can conclude from   Proposition~\ref{stat}(a) and \cite[Theorem~4.7(a)]{ms2} that $\lm_i$ is parabolically epi-differentiable at $X$ for any $H\in C_{\lm_i}(X,Y)$ in the sense of \cite[Definition~13.59]{rw}. This, coupled with Proposition~\ref{prop:dom_ssub},, ensures the claim about the domain of 
    $\d^2\lm_i(X,Y)$. 
    
\noindent{\it Claim 2:} The second subderivative of $\lm_i$ at $X$ for $Y$ has a representation of the form 
$$
\d^2g(X,Y)(H)=\dd_{C_g(X,Y)}(H)+2 \sum_{s=1}^r\big\la \Lm(Y)_{\al_s\al_s}, U_{\al_s}^\top H ( \mu_{s}I- X)^{\dagger} H  U_{\al_s}\big\ra
$$
for any $H\in \S^n$.

To verify this claim, it follows from $Y\in \sub \lm_i(X)$ and Proposition~\ref{fop_lead} that $Y=U_{\al_i}\Diag(y)_{\al_i\al_i}U_{\al_i}^\top$ for some $y\in \sub\phi_i(\lm(X))$. Recall from Claim 1 that  $\dom \d^2\lm_i(X,Y)=C_{\lm_i}(X,Y)$. Take $H\in \dom \d^2g(X,Y)$ and conclude that $\d^2g(X,Y)(H)<\infty$. Using the formula for the second suderivative from Theorem~\ref{thm:ted} then leads us to $\lm'(X;H)\in \dom \d^2\th\big(\lm(X),y\big)$. On the other hand, we know from Proposition~\ref{stat}(a) that $\d^2\th\big(\lm(X),y\big)=\d^2h\big(\lm(X),y\big)=\dd_{C_{\th}(\lm(X),y)}$, where $h(z):=\max\{z_i\,|\, i\in I(\lm(X))\}$ for any $z=(z_1,\ldots,z_n)$, and where the last equality results from $h$ being a polyhedral function together with \cite[Proposition~13.9]{rw}. Therefore,  we get  $\d^2\th\big(\lm(X),y\big)\big(\lm'(X;H)\big)=0$. 
combining this with the formula for the second suderivative from Theorem~\ref{thm:ted} confirms that for any $H\in \dom \d^2g(X,Y)$, we
have 
$$
\d^2g(X,Y)(H)=2 \sum_{s=1}^r\big\la \Lm(Y)_{\al_s\al_s}, U_{\al_s}^\top H ( \mu_{s}I- X)^{\dagger} H  U_{\al_s}\big\ra.
$$
If $H\notin \dom \d^2g(X,Y)$, we conclude from the equality 
$\dom \d^2\lm_i(X,Y)=C_{\lm_i}(X,Y)$ that $\d^2\lm_i(X,Y)(H)=\infty$, which proves the claimed formula for $\d^2\lm_i(X,Y)$ in this claim. 

To finish the proof, recall that  $Y=U_{\al_i}\Diag(y)_{\al_i\al_i}U_{\al_i}^\top$ for some $y\in \sub\phi_i(\lm(X))$. This  implies that $\Lm(Y)_{\al_s\al_s}=0$ for all $s\in \{1,\ldots,r\}\setminus \{\al_i\}$.  Combining these and Claim 2,
    we arrive at 
 \begin{eqnarray*}
 \d^2g(X,Y)(H)&=& \dd_{C_g(X,Y)}(H) +2 \big\la \Lm(Y)_{\al_i\al_i}, U_{\al_i}^\top H ( \mu_{i}I- X)^{\dagger} H  U_{\al_i}\big\ra\nonumber\\
 &=& \dd_{C_g(X,Y)}(H)+ 2\big\la Y, H ( \mu_{i}I- X)^{\dagger} H  \big\ra 
\end{eqnarray*}
for any $H\in \S^n$, which completes the proof. 
\end{proof}

It is worth  mentioning here  that twice epi-differentiability of the leading eigenvalue functions was studied before in \cite[Theorem 3.2]{t2}. Here, we obtain it via a different approach, which is much simpler than deriving it from the definition directly.  

\subsection{Second-Order Variational Properties of Statistic Regularizer}
We continue to study second-order variational properties of some important regularization terms, used often in statistics to analyze  eigenvalue optimization problems such as the largest eigenvalue gaps. This again  demonstrates the great potential of our established theory  in practical  applications. 
\begin{Example}[largest eigenvalue gap] {\rm Suppose that $X\in \S^n$. The largest eigenvalue gap of
$X$ is defined by 
$$
g(X)=\max\big\{\lambda_k(X)-\lambda_{k+1}(X)\big |\; k=1,\ldots,n\big\}.
$$
 The largest eigenvalue gap of the Laplacian matrix is frequently used in statistics to choose    the number $k$ of clusters in spectral clustering; see \cite{msx} and \cite[Section 8.3]{Von} for more details.  It is also widely used in quantum mechanics to determine the mixing time of quantum walks \cite{TOVPV}. 
 It is easy to see that $g$ has the spectral representation in \eqref{spec} with 
  $$
  \theta(x):=\max\big\{\langle x,e^{i,i+1}\ra|\; i=1,\ldots,n-1\},
  $$ 
  where $e^{i,i+1}\in\mathbb{R}^n$ is the vector whose $i$-th component is 1, $i+1$-th component is -1 and the remaining components are 0. Note that $\theta$ is the fussed graphical Lasso regularization term and its proximal mapping was  computed in \cite{ZZST}.  Since    $\theta$ is a  polyhedral function, it follows from \cite[Proposition 13.9]{rw} that such $\theta$ is twice epi-diﬀerentiable at any $x\in \R^n$ for any $y\in \sub \th(x)$. Using Theorem \ref{thm:ted}, we conclude  that $g$ is also twice epi-diﬀerentiable at $X$ for any $Y\in\partial g(X)$. 
It is also worth mentioning that our argument above can be applied to any polyhedral function  $\theta$. This tells us a similar conclusion can be reached about twice epi-differentiability of a wide range of functions appearing in different optimization problems such as  positive semidefinite programming problems  and the fastest mixing Markov chain problem \cite{BDXiao}. In such a case, according to Corollary \ref{coro:prox-d-d},  $\prox_g$ is directionally differentiable, which allows us to apply fast algorithms like the generalized Newton method and forward-backward envelop acceleration \cite{KMP23} to solve the optimization problem with such a regularizer $g$. 
}
\end{Example}

\begin{Example}[Minimax Concave Penalty (MCP)] {\rm The MCP function was introduced   in \cite[equation (2.1)]{Zhang} for the sparse regression. A frequently used MCP function (cf. \cite{Zhang, LTY}) has the   form $$\phi(t)=\begin{cases}
        c|t|-\frac{t^2}{2a} & \mbox{if}\;|t|\leq a c,\\
        \frac{a c^2}{2} & \mbox{if}\;|t|>a c,
    \end{cases}$$
    where $a>1$ and $c>0$ are two given  constants and $t\in \R$. 
    It is usually regarded as a smoothing function for the $l_0$-norm \cite{GJM, Zhang}. It can also be composed with an eigenvalue function to form a spectral function in applications. For instance,  in the robust PCA (cf. \cite{LTY}), it is common to consider the spectral function $g(X)=\sum_{i=1}^n\phi(\lambda_i(X))$.  Indeed, all the penalty functions in \cite[Table 1]{LTY} can be used in a similar way to construct a spectral function to deal with the  low rank optimization problems. Define the function $\th:\R^n\to \R$ by 
    $$
    \th(x_1,\ldots,x_n)=\phi(x_1)+\ldots+\phi(x_n).
    $$
    Clearly, $\th$ is symmetric and $g=\th\circ\lm$, which tells us that $g$ is a spectral function.
     It is easy to see that $\theta$ is differentiable at any $x=(x_1,\ldots,x_n)\in \R^n$ when $x_i\neq 0$ for all $i=1,\ldots,n$. At $x=0$, $\th$ enjoys the following properties:  
    \begin{itemize}
    \item $\theta$ is Lipschitz continuous around $0$ and $\partial\phi(0)=[-c,c]$. This, along with \cite[Propostion~10.5]{rw}, shows that 
    $\sub \th(0)=[-c,c]^n$.
        \item Since $\phi$ can be equivalently expressed as $\phi(t)=\max\{ct-\frac{t^2}{2a}, -ct-\frac{t^2}{2a}\}$ locally around $0$, it  is fully amenable at $0\in\R$ in the sense of  \cite[Definition 10.23]{rw}. According to \cite[Example~10.26]{rw}, $\th$ enjoys this property at $0\in \R^n$ as well and therefore  is prox-regular and subdifferential continuous at $0 $. 
        
        \item By  \cite[Example 7.28]{rw},   $\phi$ is subdifferentially regular at $0$. It follows then from \cite[Proposition~10.5]{rw} that 
        $\th$ is subdifferentially regular at $0\in\R^n$.
        Employing now \cite[Corollary~4]{l99}, we can conclude  that $g$ is also subdifferentially regular at $0\in \S^n$. 
        \item It follows from \cite[Corollary 13.15]{rw} that $\theta$ is twice epi-differentiable at $0\in\R$ for any $v\in \sub \phi(0)$ and 
        $$\d^2\phi(0,v)(w)=\delta_{C_{\theta}(0,v)}(w)-\frac{w^2}{a},$$
        where the last equality comes from \cite[Example 13.16]{rw}. Note that the critical cone $C_{\phi}(0,v)$ can be efficiently calculated using the representation $C_{\phi}(0,v)=N_{\sub \phi(0)}(v)$. 
        Because of the definition of $\th$, the latter tells us that $\th$ is twice epi-differentiable at $0\in \R^n$ for any $y=(y_1,\ldots,y_n)\in \sub \th(0)$ and 
        $$
        \d^2\theta(0,y)(w)=\d^2\phi(0,y_1)(w_1)+\ldots+ \d^2\phi(0,y_n)(w_n),
        $$
        for any $w=(w_1,\ldots,w_n)\in\R^n$.
    \end{itemize}
Combining these properties and Theorem \ref{thm:ted} tells us  that the spectral function $g$ is  twice epi-differentiable
at $X:=0$ for any $Y\in \sub g(X)$. A similar argument can be made for any $x=(x_1,\ldots,x_n)\in \R^n$ when one of $x_i=0$. 
Note that  for any $ x\in\R$, it is easy to see that for sufficiently small $\gamma$ that $${\rm prox}_{\gamma\phi}(x)=\begin{cases}
    0, & x\leq \gamma c,\\
   {\rm sgn}(x)\cdot\frac{a}{a-\gamma}(|x|-\gamma c) & \gamma c< x\leq a c,\\
   x & x>a c.
\end{cases}$$
It can be checked directly that ${\rm prox}_{\gamma\theta}(x)=({\rm prox}_{\gamma\phi}(x_1),\dots,{\rm prox}_{\gamma\phi}(x_n))$. 
Since $\phi$ is prox-bounded, 
we derive from Corollary \ref{coro:prox-d-d} that ${\rm prox}_{\gamma g}$ is directionally differentiable, which can be used to design efficient methods
for solving eigenvalue optimization problems using the generalized Newton method in \cite{KMP23}.}
\end{Example}

\section{Conclusion}

In this paper, we explore the second-order variational properties of a spectral function whose symmetric part is not necessarily convex. We prove    twice epi-differentiability of  spectral functions amounts to that of their symmetric parts.   By applying this theoretical equivalence as a tool, we   characterize several variational properties of  spectral functions inclduing generalized twice differentiability, twice semidifferentiability,  proto-differentiability, and directional differentiability of their proximal mappings.  As an immediate consequence, we  obtain   twice epi-differentiability of the leading eigenvalue functions and the largest eigenvalue gap. The established theory can be leveraged to solve practical optimization problems using the generalized Newton method or forward-backward envelop acceleration method, which is a subject for our future research.
It is  very likely that our second-order analysis in this paper can be extended to singular value spectral functions and it is a subject of our ongoing research.


\end{document}